\documentclass{amsart}

\usepackage{amsmath,amssymb}
\usepackage{amsthm}
\usepackage{amsrefs}
\usepackage{indentfirst}
\usepackage{mathrsfs}
\usepackage{hyperref}
\usepackage{xcolor}
\usepackage[margin=1in]{geometry}
\usepackage{nicefrac}
\usepackage{enumerate}
\usepackage{graphicx}
\usepackage{subcaption}
\usepackage{booktabs}
\usepackage{multirow}
\usepackage{graphicx}

\newlength{\leftstackrelawd}
\newlength{\leftstackrelbwd}
\def\leftstackrel#1#2{\settowidth{\leftstackrelawd}%
	{${{}^{#1}}$}\settowidth{\leftstackrelbwd}{$#2$}%
	\addtolength{\leftstackrelawd}{-\leftstackrelbwd}%
	\leavevmode\ifthenelse{\lengthtest{\leftstackrelawd>0pt}}%
	{\kern-.5\leftstackrelawd}{}\mathrel{\mathop{#2}\limits^{#1}}}

\numberwithin{equation}{section}
\newtheorem{Theorem}{Theorem}[section]
\newtheorem{lemma}[Theorem]{Lemma}
\newtheorem{remark}[Theorem]{Remark}

\newtheorem{assumption}[Theorem]{Assumption}
\newtheorem{defi}[Theorem]{Definition}
\newtheorem{cor}[Theorem]{Corollary}
\allowdisplaybreaks[4]

\newcommand{\R}{\mathbb{R}}

\newcommand{\N}{\mathbb{N}}
\newcommand{\argmin}{\mathop{\arg\min}}

\newcommand{\region}{\mathcal{D}}
\newcommand{\diff}{\mathop{}\!\mathrm{d}}
\newcommand{\mM}{\mathcal{M}}
\newcommand{\mT}{\mathcal{T}}
\newcommand{\mF}{\mathcal{F}}
\newcommand{\mE}{\mathcal{E}}
\providecommand{\dual}[1]{\left\langle#1\right\rangle}
\providecommand{\norm}[1]{\left\lVert#1\right\rVert}
\providecommand{\abs}[1]{\left\lvert#1\right\rvert}

\title[Finite Element Approximation for Gross--Pitaevskii Energy Functional]{Nonconforming Finite Element Approximation and Energy Lower Bound Estimation for the Gross--Pitaevskii Energy Functional}
\author{Chen Zhang}
\address{(CZ) School of Mathematical Sciences, Fudan University, Shanghai, 200433, People's Republic of China.}
\email{chenzhang24@m.fudan.edu.cn}
\author{Heyan Zhu}
\address{(HZ) School of Mathematical Sciences, Fudan University, Shanghai, 200433, People's Republic of China.}
\email{hyzhu25@m.fudan.edu.cn}
\author{Wenbin Chen}
\address{(WC) School of Mathematical Sciences and Shanghai Key Laboratory for Contemporary Applied Mathematics, Fudan University, Shanghai, 200433, People's Republic of China.}
\email{wbchen@fudan.edu.cn}

\begin{document}

    \begin{abstract}

	    The ground state of Bose--Einstein condensates can be described as the minimizer of the Gross--Pitaevskii energy functional subject to a mass conservation constraint. In this paper, we study the corresponding discrete optimization problem in nonconforming finite element spaces and establish a priori error estimates for the discrete ground state energy, the discrete eigenvalue, and the discrete ground state. Specifically, we derive explicit convergence rates for the a priori error in the particular case of the $EQ_1^{\mathrm{rot}}$ finite element. Furthermore, we proof that within the $EQ_1^{\mathrm{rot}}$ finite element framework, the discrete ground state energy provides a lower bound estimation to the exact energy. Finally, numerical experiments are presented to validate the theoretical analysis. 
    
    \end{abstract}

    \keywords{Gross--Pitaevskii energy functional, Bose--Einstein condensate, nonconforming finite element.}
        
    \maketitle

    \section{Introduction}

        At ultra-low temperatures, dilute bosonic gases can transition into a unique state of matter known as a Bose-Einstein condensate (BEC)~\cites{alma991054761079706011,bose1924plancks,einstein2005quantentheorie}. Within this context, the ground state—defined as the state of lowest energy—represents a central problem of study and has garnered significant attention. Mathematically, the ground state is characterized as a minimizer of the Gross--Pitaevskii energy functional over the $L^2$ sphere, where the energy functional takes the form as following~\cites{bao2012mathematical,bao2014mathematical}: for any $u \in \mathcal{M}$, 
        \begin{equation*}
            E(v) := \frac{1}{2}\int_{\region} \left(\abs{\nabla v}^2 + V\abs{v}^2 + \frac{\beta}{2}\abs{v}^4\right) \diff r. 
        \end{equation*}
        The $L^2$ sphere $\mM$ is defined as
        \begin{equation*}
            \mM := \{v\in H_0^1(\region): \norm{v}_{L^2(\region)}=1\},
        \end{equation*}
        where the integration domain $\mathcal{D}\subseteq \mathbb{R}^d$ ($d=2,3$) is a bounded convex region, the domain of the functional is $H_0^1(\mathcal{D})$, the function $V\in L^{\infty}(\region;\R_{\ge 0})$ represents the trapping potential, and the constant $\beta\in \R_{\ge 0}$ describes interparticle repulsion. The $L^2$ spherical condition can be regarded as a constraint for the mass of the condensates. For fundamental theories and mathematical properties of the ground state, we refer to~\cites{bao2005ground,bao2012mathematical,bao2014mathematical}.

        We denote the ground state function $u\in \mM$ be a minimizer of the Gross--Pitaevskii energy functional. By considering the constrained Euler--Lagrange equations, the optimization problem for the BEC ground state can be reformulated as the Gross--Pitaevskii eigenvalue problem:
        \begin{equation*}
            -\Delta u + Vu + \beta \abs{u}^2 u = \lambda u. 
        \end{equation*}
        There are many papers studying the numerical algorithms of the Gross--Pitaevskii equation, including approaches such as self-consistent field iterations~\cite{dion2007ground}, generalized inverse iterations~\cites{jarlebring2014inverse,altmann2021j}, $L^2$ gradient flows~\cite{bao2004computing}, Sobolev gradient flows~\cites{danaila2010new,kazemi2010minimizing,chen2024convergence,henning2020sobolev,zhang2025convergence,chen2023second}, Newton-type methods~\cites{xu2021multigrid,wu2017regularized}, and other iterative schemes~\cites{antoine2017efficient,bao2010efficient,bao2008generalized}. Beyond iterative solvers, it is essential to consider appropriate spatial discretization methods, which constitute the primary focus of this work. Typical discretization techniques include spectral and pseudo-spectral methods~\cites{bao2012mathematical,bao2014mathematical,cances2010numerical}, spectral element methods~\cite{chen2024fully}, Lagrange finite elements~\cites{cances2010numerical,henning2025discrete}, adaptive finite elements~\cites{chen2011adaptive,danaila2010finite,heid2021gradient}, multigrid methods~\cites{zhang2019efficient,xie2016multigrid,cances2018two}, discontinuous Galerkin composite finite elements~\cite{engstrom2022higher}, and hybrid high-order methods~\cite{hauck2025hybrid}.

        The existing theoretical analysis for these spatial discretizations has focused mainly on conforming finite element methods, which require the finite element space $V_h$ to be a subspace of $H_0^1(\region)$. Zhou~\cites{zhou2003analysis,zhou2007finite} provided approximation estimates for arbitrary finite-dimensional subspaces of $H_0^1(\region)$ along with $L^2$ a priori error estimates. Canc\`es et al.~\cite{cances2010numerical} presented a complete error analysis for the same problem, including a priori estimates for the $L^2$-norm, $H^1$-norm, energy, and ground state eigenvalue. These results were later extended by Chen et al.~\cite{chen2011finite} by considering the Hartree-type potentials. Furthermore, error estimates for the Kohn--Sham potential were obtained in~\cites{chen2014adaptive,chen2013numerical} by Chen et al., and finite element estimates for the Hartree--Fock model were derived by Maday and Turinici \cite{maday2003error}. Henning et al.~\cites{henning2014two,henning2023optimal} investigated generalized finite element approximations for the Gross--Pitaevskii equation. Recently, Henning and Yadav~\cite{henning2025discrete} studied the approximations of rotating Bose--Einstein condensates in $\mathbb{P}^k$-Lagrange finite element spaces.

        A characteristic of conforming finite element methods is that the discrete ground state energy approximates the exact value from above, as the minimization is performed over a subspace of $H_0^1(\region)$. In contrast, a significant property of some nonconforming finite element methods is their ability to provide lower bound estimates for eigenvalues in linear elliptic eigenvalue problems~\cites{lin2007finite,yang2000posteriori,yang2010eigenvalue,zhimin2007eigenvalue,hu2014lower,liang2022weak}. Recently, Hauck and Liang~\cite{hauck2025hybrid} extended this conclusion, demonstrating that hybrid high-order methods can yield lower bound approximations for the Gross--Pitaevskii energy minimum. However, a corresponding theoretical analysis for standard nonconforming finite elements is still lacking. In this paper, following the methodology of Canc\`es et al.~\cite{cances2010numerical}, we establish the approximation properties of BEC ground states in general nonconforming finite element spaces. This includes a priori error estimates for the energy, eigenvalue, and the $L^2$ and $H^1$ errors of the ground state. We then focus on the $EQ_1^{\mathrm{rot}}$ finite element, deriving explicit convergence rates within the genral approximation properties. Furthermore, drawing on the approaches of Hu et al.~\cite{hu2014lower} and Li~\cite{li2008lower}, we demonstrate that the $EQ_1^{\mathrm{rot}}$ element provides a lower bound approximation to the Gross--Pitaevskii energy minimum. These theoretical results are subsequently validated by numerical experiments.

        The remainder of this paper is organized as follows. In Section~\ref{sec:fem}, we introduce the nonconforming finite element spaces and state the assumptions required for our analysis, alongside two lemmas concerning the ground state. Section~\ref{sec:asymptotics} is devoted to the asymptotic convergence of the discrete ground state $u_h$, the discrete energy $E(u_h)$, and the eigenvalue $\lambda_h$, as well as error estimates for the latter two quantities. In Section~\ref{sec:prior}, we derive a priori $H^1$ and $L^2$ error estimates for the discrete ground state $u_h$ relative to the exact ground state $u$. In Section~\ref{sec:lower}, we consider the specific case of the $EQ_1^{\mathrm{rot}}$ element space, providing concrete error estimates and proving its lower bound property for the energy minimum. Finally, numerical experiments are presented in Section~\ref{sec:numerical} to support our theoretical results. Several lemmas related to nonconforming finite element spaces are provided in Appendix~\ref{sec:prepare}, and the verification that the $EQ_1^{\mathrm{rot}}$ space satisfies our assumptions is given in Appendix~\ref{sec:eq1rot}.

    \section{Nonconforming finite element spaces}\label{sec:fem}

        Consider the constrained optimization problem for the Gross--Pitaevskii energy functional:
        \begin{equation}\label{eq:gp}
            E(u) = \inf_{v\in \mM} E(v) = \inf_{v\in \mM} \frac{1}{2}\int_{\region} \left( \abs{\nabla v}^2 + V\abs{v}^2 + \frac{\beta}{2}\abs{v}^4 \right) \diff r, 
        \end{equation}
        where the computational domain $\region\subseteq \R^d$ ($d=2,3$) is a bounded convex region. The constraint manifold $\mM=\{v\in H_0^1(\region):\norm{u}_{L^2(\region)}=1\}$ represents the mass conservation condition, $V\in L^{\infty}(\region;\R_{\ge 0})$ is the trapping potential, and the constant $\beta>0$ describes the repulsive interactions between particles. We denote by $\langle\cdot,\cdot\rangle$ the dual pairing between $H_0^1(\mathcal{D})$ and its dual space $H^{-1}(\mathcal{D})$. The linear operator $A_v: H_0^1(\region) \to H^{-1}(\region)$ is defined as 
        \begin{equation}\label{eq:A}
            A_v w := -\Delta w + Vw + \beta \abs{v}^2 w,
        \end{equation}
        and let $E'(v) \in H^{-1}(\region)$ denote the Fréchet derivative of the functional $E(v)$. Following the standard approach for constrained optimization problems (cf. \cite{chen2024convergence}), we introduce a Lagrange multiplier $\lambda$ and rewrite the constraint as $\frac{1}{2}(\norm{u}_{L^2(\region)}^2-1)=0$, which yields the Euler--Lagrange equations:
        \begin{equation*}
            (\nabla u, \nabla v)_{L^2(\region)} + (Vu + \abs{u}^2 u, v)_{L^2(\region)} = \langle A_u u,v \rangle = \langle E'(u),v \rangle = \lambda (u,v)_{L^2(\region)}\qquad \forall v\in H_0^1(\region), 
        \end{equation*}
        or equivalently, 
        \begin{equation}\label{eq:gpe}
            -\Delta u + Vu + \beta \abs{u}^2 u = A_u u = E'(u) = \lambda u. 
        \end{equation}
        In equations \eqref{eq:A} and \eqref{eq:gpe}, the operator $\Delta$ is understood in the distributional sense, i.e., 
        \begin{equation*}
            \dual{-\Delta v,w} := \int_{\region} \nabla v \cdot \nabla w \diff r, \qquad \forall v,w\in H_0^1(\region). 
        \end{equation*}
        We refer to \eqref{eq:gpe} as the Gross--Pitaevskii eigenvalue problem, where $\lambda$ is the eigenvalue associated with the eigenfunction $u$.

        The minimizers of the Gross--Pitaevskii energy functional are termed the ground states of the Bose--Einstein condensate. Canc\`es et al. established that a solution to the constrained optimization problem~\eqref{eq:gp} exists and is unique up to a sign~\cite{cances2010numerical}*{Lemma 2}. Let $u:=\argmin\limits_{v\in \mM} E(v)$ denote a ground state and $\lambda$ the corresponding ground state eigenvalue. Based on standard elliptic regularity theory, Henning and Yadav proved the following result~\cite{henning2025discrete}*{Lemma 2.5}.

        \begin{lemma}\label{lem:reg}
            The ground state $u$ possesses $H^2$-regularity, i.e., $u\in C^0(\overline{\region})\cap H^2(\region)$. 
        \end{lemma}

        For the optimization problem~\eqref{eq:gp}, we consider an approximation using nonconforming finite element spaces. Let $\region$ be a polygonal domain and $\mT_h$ a regular partition of $\region$. Let $V_h$ denote the corresponding nonconforming finite element space. For a given positive integer $k \in \mathbb{N}_+$, any $v_h \in V_h$ satisfies $v_h|_K \in H^1(K) \cap \mathbb{P}^k(K)$ for each $K \in \mT_h$, although $V_h \not\subseteq H_0^1(\region)$. We define the discrete constraint manifold $\mM_h$ as 
        \begin{equation*}
            \mM_h = \{v_h\in V_h:\norm{v_h}_{L^2(\region)}=1\},
        \end{equation*}
        and the piecewise gradient operator $\nabla_h$ such that
        \begin{equation*}
            (\nabla_h v_h)|_K = \nabla (v_h|_K). 
        \end{equation*}
        Furthermore, the energy $E$ and the operator $A_v$ are extended to $V_h$ as follows:
        \begin{align*}
            &E(v_h) = \frac{1}{2}\int_{\region} \left( \abs{\nabla_h  v_h}^2 + V\abs{v_h}^2 + \frac{\beta}{2}\abs{v_h}^4 \right) \diff r, \\
            &\dual{A_v v_h, w_h} = \int_{\region} \left( \nabla_h v_h\cdot \nabla_h w_h + Vv_h w_h + \beta|v|^2 v_h w_h \right) \diff r. 
        \end{align*}
        These definitions are well-posed since $\nabla v_h = \nabla_h v_h$ whenever $v_h \in V_h \cap H_0^1(\region)$. We consider the discrete constrained optimization problem
        \begin{equation}\label{eq:gph}
            E(u_h) = \inf_{v_h\in \mM_h} E(v_h), 
        \end{equation}
        By Lagrange multipliers, this problem is equivalent to the discrete eigenvalue problem:
        \begin{equation}\label{eq:gpeh}
            \dual{A_{u_h} u_h, v_h} = \lambda_h (u_h, v_h)_{L^2(\region)}, \qquad \forall v_h\in V_h. 
        \end{equation}
        Since $\mM_h$ is a compact manifold in a finite-dimensional space, the energy functional $E$ attains its minimum on $\mM_h$. We denote a solution to \eqref{eq:gph} by $u_h$ and the corresponding eigenvalue in \eqref{eq:gpeh} by $\lambda_h$. These are referred to as the discrete ground state and the discrete ground state eigenvalue, respectively. For any $v_h, w_h \in V_h + H_0^1(\region)$, we define the broken $H^1$-inner product, norm, and semi-norm on the space $V_h + H_0^1(\region)$ as
        \begin{align*}
            &(v_h, w_h)_{H^1_h(\region)} = (v_h, w_h)_{L^2(\region)} + (\nabla_h v_h, \nabla_h w_h)_{L^2(\region)}, \\
            &\abs{v_h}_{H^1_h(\region)} = \norm{\nabla_h v_h}_{L^2(\region)}, \qquad \norm{v_h}_{H^1_h(\region)} = (v_h, v_h)_{H^1_h(\region)}^{\frac{1}{2}}. 
        \end{align*}
        As the space $V_h \not\subseteq H_0^1(\region)$, integration by parts on $V_h$ gives rise to a consistency error~\cite{ciarlet2002finite}. We define the consistency error functional $\mE_h: (V_h + H_0^1(\region))\times H^1(\region; \R^d) \to \R$ as 
        \begin{equation*}
            \mE_h(v_h, v) := \int_{\region} \left( v_h(\nabla \cdot v) + v\cdot \nabla_h v_h \right) \diff r = \sum_{K\in \mT_h}\int_{\partial K}v_h (v\cdot n) \diff S. 
        \end{equation*} 
        where $n$ denotes the unit outward normal vector. Let $\mF_h$ be the set of all interfaces in the partition $\mT_h$, and let $[v_h]$ denote the jump of $v_h$ across these interfaces, defined by
        \begin{align*}
            &[v_h]|_{e} := v_h|_{K}, \qquad \forall e\subseteq \partial \region \cap \partial K, e \in \mF_h, \\
            &[v_h]|_{e} := (v_h|_{K_1} - v_h|_{K_2})|_{e}, \qquad \forall e\nsubseteq \partial \region, e = K_1\cap K_2 \in \mF_h. 
        \end{align*}
        Throughout this paper, the notation $a \lesssim b$ indicates that $a \le Cb$ for a constant $C > 0$ that is independent of the mesh size $h$, but may depend on the domain $\region$, the potential $V$, the constant $\beta$, the ground state $u$, the ground state eigenvalue $\lambda$, the mesh regularity of $\mT_h$, and the polynomial degree $k$. To facilitate our analysis, we assume that the nonconforming finite element space $V_h$ satisfies the following conditions~\cite{shi2010finite}. Note that condition~\eqref{eq:assumpt:3} in Assumption~\ref{assumpt:consist} is more restrictive than the original generalized patch test~\cite{stummel1979generalized}, which only requires $v_h \in V_h$.

        \begin{assumption}\label{assumpt:consist}
            We make the following assumptions on the finite element space $V_h$:
            \begin{enumerate}
                \item $V_h$ possesses the approximation property, i.e., for any $v\in H_0^1(\region)$, 
                    \begin{equation}\label{eq:assumpt:1}
                        \lim_{h\to 0} \inf_{v_h\in V_h} \norm{v-v_h}_{H^1_h(\region)} = 0. 
                    \end{equation}
                \item The pair $\{V_h, H_0^1(\region)\}$ passes the generalized patch test, i.e., for any $v\in H^1(\region; \R^d)$, 
                    \begin{equation}\label{eq:assumpt:3}
                        \lim_{h\to 0} \sup_{\substack{v_h\in V_h + H_0^1(\region) \\ \norm{v_h}_{H^1_h(\region)} \le 1}} \abs{\mE_h(v_h, v)} = 0. 
                    \end{equation}
                \item For any $v_h\in V_h + H_0^1(\region)$, 
                    \begin{equation}\label{eq:assumpt:4}
                        \left( \sum_{e\in \mF_h} h^{-1}\norm{[v_h]}_{L^2(e)}^2 \right)^{\frac{1}{2}} \lesssim \norm{v_h}_{H^1_h(\region)}. 
                    \end{equation}
            \end{enumerate}
        \end{assumption}

        Let the second order Fréchet derivative of the energy functional $E(v)$, denoted by $E''(v): (V_h + H_0^1(\region)) \to (V_h + H_0^1(\region))'$, be given by 
        \begin{equation*}
            E''(v)w = A_v w + 2\beta \abs{v}^2 w, 
        \end{equation*}
        Regarding the second order Fréchet derivative $E''(u)$ evaluated at the ground state $u$, we have the following lemma.

        \begin{lemma}\label{lem:coerc}
            For a sufficiently small mesh size $h$, the operator $E''(u) - \lambda$ is bounded and coercive on $V_h + H_0^1(\region)$. Specifically, there exist positive constants $\mu, M>0$ such that 
            \begin{align*}
                &\dual{\left(E''(u)-\lambda\right)v_h, w_h} \le M\norm{v_h}_{H^1_h(\region)}\norm{w_h}_{H^1_h(\region)}, \qquad \forall v_h, w_h\in V_h + H_0^1(\region), \\
                &\dual{\left(E''(u)-\lambda\right)v_h, v_h} \ge \mu\norm{v_h}_{H^1_h(\region)}^2, \qquad \forall v_h\in V_h + H_0^1(\region).
            \end{align*}
        \end{lemma}

        \begin{proof}
            For boundedness, we choose $M = 1 + \norm{V}_{L^{\infty}(\region)} + 3\beta \norm{u}_{L^{\infty}(\region)}^2 + \lambda$, which leads to
            \begin{equation*}
                \begin{aligned}
                    &\dual{\left(E''(u)-\lambda\right)v_h, w_h} \\
                    \le& \abs{v_h}_{H^1_h(\region)}\abs{w_h}_{H^1_h(\region)} + \left(\norm{V}_{L^{\infty}(\region)} + 3\beta \norm{u}_{L^{\infty}(\region)}^2 + \lambda\right)\norm{v_h}_{H^1_h(\region)}\norm{w_h}_{H^1_h(\region)} \\
                    \le& M\norm{v_h}_{H^1_h(\region)}\norm{w_h}_{H^1_h(\region)}. 
                \end{aligned}
            \end{equation*}
            To prove coercivity, we proceed by contradiction. Suppose that for any $h > 0$ and any $\delta > 0$, there exist a mesh size $h_{\delta} \le h$ and a function $v_{h_{\delta}}\in V_{h_{\delta}} + H_0^1(\region)$ such that 
            \begin{equation*}
                \dual{\left(E''(u)-\lambda\right)v_{h_{\delta}}, v_{h_{\delta}}}\le \delta \norm{v_{h_{\delta}}}_{H^1_h(\region)}^2, 
            \end{equation*}
            Then, one can find a sequence of mesh sizes $\{h_j\}_{j\in \mathbb{N}}$ and a corresponding sequence of functions $\{v_{h_j}\}_{j\in \mathbb{N}}$ satisfying 
            \begin{equation*}
                \lim_{j\to \infty}h_j = 0,\qquad \norm{v_{h_j}}_{H^1_h(\region)} = 1,\qquad \limsup_{j\to \infty} \dual{\left(E''(u)-\lambda\right)v_{h_j}, v_{h_j}} \le 0. 
            \end{equation*}
            By Lemma~\ref{lem:embd} and the Banach--Alaoglu theorem, there exist $v\in L^2(\region)$ and $\tilde{v}\in L^2(\region; \R^d)$, and a subsequence of $\{v_{h_j}\}$ (still denoted by $\{v_{h_j}\}$), such that  
            \begin{align*}
                &v_{h_j} \to v, \qquad \text{strongly in $L^2(\region)$}, \\
                &\nabla_{h_j} v_{h_j} \to \tilde{v}, \qquad \text{weakly in $L^2(\region)$}. 
            \end{align*}
            For any $\varphi\in C_0^{\infty}(\R^d; \R^d)$, we have 
            \begin{equation*}
                \begin{aligned}
                    \int_{\region} \tilde{v} \cdot \varphi \diff r &= \lim_{j\to \infty} \int_{\region} \nabla_{h_j} v_{h_j} \cdot \varphi \diff r\\
                    &= \lim_{j\to \infty} \left( -\int_{\region} v_{h_j} \left(\nabla \cdot \varphi\right) \diff r + \sum_{K\in \mT_{h_j}} \int_{\partial K} v_{h_j} \left(\varphi \cdot n\right) \diff S \right)\\
                    &= -\int_{\region} v \left(\nabla \cdot \varphi\right) \diff r,
                \end{aligned}
            \end{equation*}
            where the final equality follows from the generalized patch test~\eqref{eq:assumpt:3}. This implies $v\in H_0^1(\region)$ and $\nabla v = \tilde{v}$. 
            Assume that $v = 0$, then 
            \begin{equation*}
                v_{h_j} \to 0, \qquad \text{strongly in $L^2(\region)$}, 
            \end{equation*}
            which indicates that 
            \begin{align*}
                &\lim_{j\to \infty} \norm{\nabla_{h_j} v_{h_j}}_{L^2(\region)} = \lim_{j\to \infty} \norm{v_{h_j}}_{H_{h_j}^1(\region)} = 1, \\
                &\lim_{j\to \infty} \int_{\region} V \abs{v_{h_j}}^2 \diff r = \lim_{j\to \infty} \int_{\region} \beta \abs{u}^2 \abs{v_{h_j}}^2 \diff r = \lim_{j\to \infty} \lambda \norm{v_{h_j}}_{L^2(\region)}^2 = 0. 
            \end{align*}
            Therefore we obtain 
            \begin{equation*}
                0 \ge \limsup_{j\to \infty} \dual{\left(E''(u)-\lambda\right)v_{h_j}, v_{h_j}} = \lim_{j\to \infty} \norm{\nabla_{h_j} v_{h_j}}_{L^2(\region)} = 1, 
            \end{equation*}
            which yields a contradiction. Thus $v \neq 0$.
            By the weak lower semi-continuity and strong continuity of the $L^2$-norm, we obtain 
            \begin{equation*}
                \dual{\left(E''(u)-\lambda\right)v, v} \le \liminf_{n\to \infty} \dual{\left(E''(u)-\lambda\right)v_{h_j}, v_{h_j}} \le \limsup_{n\to \infty} \dual{\left(E''(u)-\lambda\right)v_{h_j}, v_{h_j}} \le 0. 
            \end{equation*}
            However, according to \cite{cances2010numerical}*{Lemma 1}, the operator $E''(u)-\lambda$ is coercive on $H_0^1(\region)$, which yields a contradiction. Thus, the coercivity is established. 
        \end{proof}

        \begin{remark}
            While Canc\`es et al. established the boundedness and coercivity of $E''(u) - \lambda$ on the space $H_0^1(\region)$ in \cite{cances2010numerical}*{Lemma 1}, we have extended this result to the broken $H^1$-space $V_h + H_0^1(\region)$. Note that coercivity implies the inf-sup condition; for further details on the properties of this condition, we refer to \cite{babuvska1971error}. 
        \end{remark}

    \section{Asymptotic convergence and error estimates for energy and eigenvalues}\label{sec:asymptotics}

        Henning and Yadav established the asymptotic convergence of the discrete ground state $u_h$, the discrete energy $E(u_h)$, and the eigenvalue $\lambda_h$ for conforming $\mathbb{P}^k$ finite elements, along with a priori error estimates for the latter two quantities~\cite{henning2025discrete}. In this section, we extend these results to the case of nonconforming finite element spaces $V_h$. We begin by stating the asymptotic convergence in Theorem~\ref{thm:asymp}.

        \begin{Theorem}\label{thm:asymp}
            For the sequence of discrete minimizers and their corresponding eigenvalues $\{(u_h, \lambda_h)\}$, there exists a ground state pair $(u, \lambda)$ and a subsequence $\{(u_{h_j}, \lambda_{h_j})\}_{j\in \mathbb{N}}$ such that
            \begin{equation*}
                \lim_{j\to \infty} E(u_{h_j}) = E(u),\qquad \lim_{j\to \infty}\norm{u-u_{h_j}}_{H^1_h(\region)} = 0,\qquad \lim_{j\to \infty}\lambda_{h_j} = \lambda. 
            \end{equation*}
            Furthermore, for any $\delta>0$, there exists a mesh size $h_{\delta}$ such that for any $h \le h_{\delta}$ and any discrete minimizer $u_h$, there exists a ground state $u$ satisfying 
            \begin{equation*}
                \norm{u-u_h}_{H^1_h(\region)}\le \delta. 
            \end{equation*}
        \end{Theorem}

        \begin{proof}
            According to Corollary~\ref{coro:bound}, the sequence $\{u_h\}$ is uniformly bounded in the $H^1_h$-norm for sufficiently small mesh sizes $h$. Consequently, by Lemma~\ref{lem:embd}, there exist $v\in L^2(\region)$ and $\tilde{v}\in L^2(\region; \mathbb{R}^d)$, and a subsequence $\{u_{h_j}\}$ such that $\lim\limits_{j\to \infty} h_j = 0$ and
            \begin{align*}
                &u_{h_j} \to v, \qquad \text{strongly in $L^2(\region)$ and $L^4(\region)$}, \\
                &\nabla_{h_j} u_{h_j} \to \tilde{v}, \qquad \text{weakly in $L^2(\region)$}. 
            \end{align*}
            Following a similar argument to the proof of Lemma~\ref{lem:coerc}, we obtain $v\in H_0^1(\region)$ and $\nabla v = \tilde{v}$. By the weak lower semi-continuity and strong continuity of the $L^2$-norm, it follows that 
            \begin{equation}\label{eq:asymp:1}
                \liminf_{j\to \infty} E(u_{h_j}) \ge E(v) \ge E(u). 
            \end{equation}
            Conversely, by virtue of Lemma~\ref{lem:proj} and the definition of $u_h$, we have 
            \begin{equation}\label{eq:asymp:2}
                \limsup_{j\to \infty} E(u_{h_j}) \le \limsup_{j\to \infty} E(P_{H^1,h_j} u) = E(u), 
            \end{equation}
            where the operator $P_{H^1,h}: H_0^1(\region)\to V_h$ is defined such that for any $v\in H_0^1(\region)$ and $v_h\in V_h$:
            \begin{equation*}
                (P_{H^1,h}v, v_h)_{H^1_h(\region)} = (v, v_h)_{H^1_h(\region)}. 
            \end{equation*}
            Combining \eqref{eq:asymp:1} and \eqref{eq:asymp:2}, we arrive at
            \begin{equation}\label{eq:asymp:3}
                \lim_{j\to \infty} E(u_{h_j}) = E(v) = E(u). 
            \end{equation}
            This indicates that $v$ is indeed a ground state, which we hereafter denote as $u$. From \eqref{eq:asymp:3} and the strong convergence of $\{u_{h_j}\}$ in $L^2(\region)$ and $L^4(\region)$, we have 
            \begin{equation}\label{eq:asymp:4}
                \lim_{j\to \infty} \abs{u_{h_j}}_{H^1_{h_j}(\region)} = \abs{u}_{H^1(\region)}, 
            \end{equation}
            which directly implies 
            \begin{equation*}
                \lim_{j\to \infty}\norm{u-u_{h_j}}_{H^1_{h_j}(\region)} = 0. 
            \end{equation*}
            Finally, since
            \begin{align*}
                &\lambda = \abs{u}_{H^1(\region)} + \norm{Vu}_{L^2(\region)} + \beta \norm{u}_{L^4(\region)}, \\
                &\lambda_{h_j} = \abs{u_{h_j}}_{H^1_{h_j}(\region)} + \norm{Vu_{h_j}}_{L^2(\region)} + \beta \norm{u_{h_j}}_{L^4(\region)}. 
            \end{align*}
            the strong convergence ensures that 
            \begin{equation*}
                \lim_{j\to \infty}\lambda_{h_j} = \lambda. 
            \end{equation*}
            The latter part of the theorem can be established by contradiction. If the statement were false, there would exist $\delta>0$ such that for any mesh size $h_{\delta}$, there exists a smaller size $\tilde{h}_{\delta} \le h_{\delta}$ where, for any ground state $u$, 
            \begin{equation}\label{eq:asymp:5}
                \norm{u-u_{\tilde{h}_{\delta}}}_{H^1_{\tilde{h}_{\delta}}(\region)}\ge \delta.
            \end{equation}
            Consequently, a subsequence $\{u_{h_j}\}$ could be extracted such that $\lim\limits_{j\to \infty} h_j = 0$ while \eqref{eq:asymp:5} holds for all ground states $u$. Repeating the previous convergence argument for this subsequence leads to a contradiction. 
        \end{proof}

        \begin{remark}
            The ground state $u$ is unique up to a sign. Thus, by imposing an appropriate condition, we obtain
            \begin{equation*}
                \lim_{h\to 0} \norm{u - u_h}_{H^1_h(\region)} = \lim_{h\to 0} \norm{u - u_h}_{L^2(\region)} = 0. 
            \end{equation*}
        \end{remark}

        The following theorems provide a priori error estimates for the eigenvalue $\lambda_h$ and the discrete energy $E(u_h)$.

        \begin{Theorem}\label{thm:eig}
            For sufficiently small mesh sizes $h$, the error between the exact ground state eigenvalue $\lambda$ and the discrete eigenvalue $\lambda_h$ satisfies
            \begin{equation*}
                \abs{\lambda - \lambda_h} \lesssim \norm{u-u_h}_{H^1_h(\region)}^2 + \norm{u-u_h}_{L^2(\region)} + \abs{\mE_h(u_h, \nabla u)}. 
            \end{equation*}
        \end{Theorem}

        \begin{proof}
            Direct calculation yields
            \begin{equation}\label{eq:eig:1}
                \begin{aligned}
                    &\dual{\left(A_u-\lambda\right)(u-u_h), u-u_h} \\
                    =& \dual{A_u u, u} - 2\dual{A_u u, u_h} + \dual{A_u u_h, u_h} - 2\lambda + 2\lambda(u, u_h)_{L^2(\region)} \\
                    =& - 2\dual{A_u u, u_h} + \dual{A_{u_h} u_h, u_h} + \left(\beta \left(|u|^2-|u_h|^2\right)u_h, u_h\right)_{L^2(\region)} - \lambda + 2\lambda(u, u_h)_{L^2(\region)} \\
                    =& \lambda_h - \lambda - 2\dual{\left(A_u-\lambda\right)u, u_h} + \left(\beta \left(|u|^2-|u_h|^2\right)u_h, u_h\right)_{L^2(\region)}.
                \end{aligned}
            \end{equation}
            By Lemma~\ref{lem:reg}, the exact ground state $u$ resides in $H_0^1(\region)\cap H^2(\region)$; thus, the eigenvalue equation~\eqref{eq:gpe} holds almost everywhere, leading to
            \begin{equation}\label{eq:eig:2}
                \dual{\left(A_u-\lambda \right)u, u_h} = \sum_{K\in\mT_h} \int_K \left(\nabla_h u_h \cdot \nabla u + u_h\Delta u\right)\diff r = \mE_h(u_h, \nabla u). 
            \end{equation}
            In a manner similar to the proof of boundedness in Lemma~\ref{lem:coerc}, we have 
            \begin{equation}\label{eq:eig:3}
                \abs{\dual{\left(A_u-\lambda\right)(u-u_h), u-u_h}} \le M \norm{u-u_h}_{H^1_h(\region)}^2. 
            \end{equation}
            Furthermore, applying the H\"older inequality and Lemma~\ref{lem:itp} gives
            \begin{equation}\label{eq:eig:4}
                \begin{aligned}
                    \left(\beta \left(\abs{u}^2-\abs{u_h}^2\right)u_h, u_h\right)_{L^2(\region)} &\le \beta \norm{u-u_h}_{L^2(\region)} \norm{u+u_h}_{L^6(\region)}\norm{u_h}_{L^6(\region)}^2 \\
                    &\lesssim \beta \norm{u-u_h}_{L^2(\region)} \left(\norm{u}_{H^1(\region)} + \norm{u_h}_{H^1_h(\region)}\right)\norm{u_h}_{H^1_h(\region)}^2. 
                \end{aligned}
            \end{equation}
            Notice that $\norm{u_h}_{H^1_h(\region)}$ is uniformly bounded for sufficiently small $h$. Substituting \eqref{eq:eig:2}--\eqref{eq:eig:4} into \eqref{eq:eig:1}, we conclude that
            \begin{equation*}
                \abs{\lambda - \lambda_h} \lesssim \norm{u-u_h}_{H^1_h(\region)}^2 + \norm{u-u_h}_{L^2(\region)} + \abs{\mE_h(u_h, \nabla u)}. 
            \end{equation*}
            This completes the proof.
        \end{proof}

        \begin{Theorem}\label{thm:energy}
            For the exact energy $E(u)$ and the discrete energy $E(u_h)$, the following estimate holds for sufficiently small mesh sizes $h$:
            \begin{equation*}
                \abs{E(u) - E(u_h)} \lesssim \norm{u - u_h}_{H^1_h(\region)}^2 + \abs{\mE_h(u_h, \nabla u)}. 
            \end{equation*}
        \end{Theorem}

        \begin{proof}
            By direct computation, we have 
            \begin{equation}\label{eq:energy:1}
                E(u_h) - E(u) = \frac{1}{2}\left( \dual{A_u u_h, u_h} - \dual{A_u u, u} \right) + \frac{\beta}{4}\int_{\region} \left(\abs{u_h}^2 - \abs{u}^2\right)^2 \diff r. 
            \end{equation}
            For the first term in \eqref{eq:energy:1}, it holds that
            \begin{equation*}
                \dual{A_u u_h, u_h} - \dual{A_u u, u} = \dual{\left(A_u - \lambda\right)(u-u_h), u-u_h} + 2\dual{\left(A_u - \lambda\right)u, u_h}. 
            \end{equation*}
            Utilizing \eqref{eq:eig:2} and \eqref{eq:eig:3}, we obtain 
            \begin{equation}\label{eq:energy:2}
                \abs{\dual{A_u u_h, u_h} - \dual{A_u u, u}} \lesssim \norm{u-u_h}_{H^1_h(\region)}^2 + \abs{\mE_h(u_h, \nabla u)}. 
            \end{equation}
            Regarding the second term in \eqref{eq:energy:1}, we observe that
            \begin{equation}\label{eq:energy:3}
                \begin{aligned}
                    \int_{\region} \left(\abs{u_h}^2 - \abs{u}^2\right)^2 \diff r &\le \int_{\region} \left(u_h - u\right)^2 \left(u_h + u\right)^2 \diff r\\
                    &\le \norm{u - u_h}_{L^4(\region)}^2 \norm{u_h + u}_{L^4(\region)}^2 \\ 
                    &\lesssim \norm{u - u_h}_{H^1_h(\region)}^2\norm{u_h + u}_{H^1_h(\region)}^2 \\
                    &\lesssim \norm{u - u_h}_{H^1_h(\region)}^2, 
                \end{aligned}
            \end{equation}
            where the final inequality follows from the uniform boundedness of $\norm{u_h}_{H^1_h(\region)}$. Substituting \eqref{eq:energy:2} and \eqref{eq:energy:3} into \eqref{eq:energy:1} yields
            \begin{equation*}
                \abs{E(u) - E(u_h)} \lesssim \norm{u - u_h}_{H^1_h(\region)}^2 + \abs{\mE_h(u_h, \nabla u)}. 
            \end{equation*}
            This completes the proof.
        \end{proof}

        \begin{remark}
            Canc\`es derived a priori error estimates for eigenvalues and discrete energies in general nonlinear eigenvalue problems and their finite-dimensional approximations~\cite{cances2010numerical}. Subsequently, Henning and Yadav extended these estimates to rotating Bose--Einstein condensates using conforming $\mathbb{P}^k$ finite elements~\cite{henning2025discrete}. In Theorems~\ref{thm:eig} and~\ref{thm:energy}, we consider the Gross--Pitaevskii eigenvalue problem without rotation and generalize these a priori error estimates to nonconforming finite element spaces.
        \end{remark}

    \section{A priori error estimates}\label{sec:prior}

        Canc\`es provided a priori error estimates in the $H^1$- and $L^2$-norms for discrete ground states $u_h$ relative to the exact ground state $u$ within the context of general nonlinear eigenvalue problems and their finite-dimensional approximations~\cite{cances2010numerical}. In this section, we extend these estimates to the nonconforming finite element approximation of the eigenvalue problem~\eqref{eq:gpe}.

        \begin{Theorem}
            Assume that the mesh size $h$ is sufficiently small and that the chosen discrete ground state satisfies $\lim\limits_{h\to 0}\norm{u - u_h}_{H^1_h(\region)} = 0$. Then, the following estimate holds:
            \begin{equation*}
                \norm{u - u_h}_{H^1_h(\region)} \lesssim \inf_{v_h\in V_h} \norm{u - v_h}_{H^1_h(\region)} + \sup_{\substack{w_h \in V_h \\ \norm{w_h}_{H^1_h(\region)}\le 1}} \abs{\mE_h\left(w_h, \nabla u\right)}. 
            \end{equation*}
        \end{Theorem}

        \begin{proof}
            For an arbitrary $v_h\in V_h$, the triangle inequality yields
            \begin{equation}\label{eq:priorH1:1}
                \norm{u - u_h}_{H^1_h(\region)}\le \norm{u_h - v_h}_{H^1_h(\region)} + \norm{u - v_h}_{H^1_h(\region)}. 
            \end{equation}
            Invoking the coercivity result from Lemma~\ref{lem:coerc}, we have
            \begin{equation}\label{eq:priorH1:2}
                \begin{aligned}
                    \norm{u_h - v_h}_{H^1_h(\region)}^2 &\le \mu^{-1} \dual{(E''(u) - \lambda)(u_h - v_h), u_h - v_h} \\
                    &\le \mu^{-1} \left(\dual{(E''(u) - \lambda)(u_h - u), u_h - v_h} + \dual{(E''(u) - \lambda)(u - v_h), u_h - v_h}\right). 
                \end{aligned}
            \end{equation}
            For any $w_h\in V_h$, direct expansion shows that
            \begin{equation}\label{eq:priorH1:3}
                \begin{aligned}
                    &\dual{(E''(u) - \lambda)(u_h - u), w_h} \\
                    =& \dual{(A_u - \lambda)(u_h - u), w_h} + \int_{\region} 2\beta \abs{u}^2 (u_h - u) w_h \diff r\\
                    =& (\lambda_h - \lambda)(u_h, w_h)_{L^2(\region)} - \sum_{K\in\mT_h} \int_{\partial K} w_h(\nabla u \cdot n) \diff S\\
                    &\quad - \int_{\region}\beta (u - u_h)^2 (2u + u_h) w_h \diff r. 
                \end{aligned}
            \end{equation}
            For $v_h\in \mM_h$, we observe that
            \begin{equation}\label{eq:priorH1:4}
                (u_h, u_h - v_h)_{L^2(\region)} = 1 - (u_h, v_h)_{L^2(\region)} = \frac{1}{2}\norm{u_h - v_h}_{L^2(\region)}^2 \le \frac{1}{2}\norm{u_h - v_h}_{H^1_h(\region)}^2. 
            \end{equation}
            Furthermore, Lemma~\ref{lem:itp} implies
            \begin{equation}\label{eq:priorH1:5}
                \abs{\int_{\region}\beta (u - u_h)^2 (2u + u_h) w_h \diff r} \lesssim \norm{u - u_h}_{H^1_h(\region)}^2\norm{w_h}_{H^1_h(\region)}. 
            \end{equation}
            Substituting \eqref{eq:priorH1:4}, \eqref{eq:priorH1:5}, and the eigenvalue estimate from Theorem~\ref{thm:eig} into \eqref{eq:priorH1:3}, we obtain for $v_h\in V_h$:
            \begin{equation}\label{eq:priorH1:6}
                \begin{aligned}
                    &\abs{\dual{(E''(u) - \lambda)(u_h - u), u_h - v_h}} \\
                    \lesssim& \left( \norm{u - u_h}_{H^1_h(\region)}^2 + \norm{u - u_h}_{L^2(\region)} + \abs{\mE_h\left(u_h, \nabla u\right)}\right)\norm{u_h - v_h}_{H^1_h(\region)}^2 \\
                    &\quad + \norm{u - u_h}_{H^1_h(\region)}^2\norm{u_h - v_h}_{H^1_h(\region)} + \abs{\mE_h\left(u_h - v_h, \nabla u\right)}. 
                \end{aligned}
            \end{equation}
            Additionally, the boundedness established in Lemma~\ref{lem:coerc} gives
            \begin{equation}\label{eq:priorH1:7}
                \abs{\dual{(E''(u) - \lambda)(u - v_h), u_h - v_h}} \le M\norm{u - v_h}_{H^1_h(\region)}\norm{u_h - v_h}_{H^1_h(\region)}. 
            \end{equation}
            Inserting \eqref{eq:priorH1:6} and \eqref{eq:priorH1:7} into \eqref{eq:priorH1:2}, for sufficiently small $h$, it follows that
            \begin{equation}\label{eq:priorH1:8}
                \begin{aligned}
                    \norm{u_h - v_h}_{H^1_h(\region)} &\lesssim \norm{u - u_h}_{H^1_h(\region)}^2 + \norm{u - v_h}_{H^1_h(\region)} + \abs{\mE_h\left(\frac{u_h - v_h}{\norm{u_h - v_h}_{H^1_h(\region)}}, \nabla u\right)} \\
                    &\le \norm{u - u_h}_{H^1_h(\region)}^2 + \norm{u - v_h}_{H^1_h(\region)} + \sup_{\substack{w_h \in V_h \\ \norm{w_h}_{H^1_h(\region)}\le 1}} \abs{\mE_h\left(w_h, \nabla u\right)}. 
                \end{aligned}
            \end{equation}
            Combining \eqref{eq:priorH1:8} with \eqref{eq:priorH1:1}, there exists a constant $C > 0$ such that 
            \begin{equation*}
                \left( 1 - C \norm{u - u_h}_{H^1_h(\region)} \right)\norm{u - u_h}_{H^1_h(\region)}\le (C + 1) \norm{u - v_h}_{H^1_h(\region)} + C \sup_{\substack{w_h \in V_h \\ \norm{w_h}_{H^1_h(\region)}\le 1}} \abs{\mE_h\left(w_h, \nabla u\right)}. 
            \end{equation*}
            Noting that $\lim_{h\to 0}\norm{u - u_h}_{H^1_h(\region)} = 0$, when the mesh size $h$ is sufficiently small, we have 
            \begin{equation}
                \norm{u - u_h}_{H^1_h(\region)} \lesssim \inf_{\substack{v_h \in V_h \\ \norm{v_h}_{L^2(\region)} = 1}} \norm{u - v_h}_{H^1_h(\region)} + \sup_{\substack{w_h \in V_h \\ \norm{w_h}_{H^1_h(\region)}\le 1}} \abs{\mE_h\left(w_h, \nabla u\right)}. 
            \end{equation}
            Following the proof of~\cite{cances2010numerical}*{Theorem 1}, the following inequality holds 
            \begin{equation}\label{eq:priorH1:9}
                \inf_{\substack{v_h \in V_h \\ \norm{v_h}_{L^2(\region)} = 1}} \norm{u - v_h}_{H^1_h(\region)} \lesssim \inf_{v_h\in V_h} \norm{u - v_h}_{H^1_h(\region)}. 
            \end{equation}
            Finally, considering the approximation property of $V_h$, we arrive at the desired estimate:
            \begin{equation*}
                \norm{u - u_h}_{H^1_h(\region)} \lesssim \inf_{v_h\in V_h} \norm{u - v_h}_{H^1_h(\region)} + \sup_{\substack{w_h \in V_h \\ \norm{w_h}_{H^1_h(\region)}\le 1}} \abs{\mE_h\left(w_h, \nabla u\right)}. 
            \end{equation*}
            The proof is complete. 
        \end{proof}

        \begin{remark}
            In \cite{cances2010numerical}*{Theorem 1}, Canc\`es et al. proved that 
            \begin{equation*}
                \inf_{\substack{v_{\delta} \in X_{\delta} \\ \norm{v_{\delta}}_{L^2(\region)} = 1}} \norm{u - v_{\delta}}_{H^1(\region)} \lesssim \inf_{v_{\delta}\in X_{\delta}} \norm{u - v_{\delta}}_{H^1(\region)}, 
            \end{equation*}
            where $X_{\delta}$ is a finite-dimensional subspace of $H_0^1(\region)$. In \eqref{eq:priorH1:9}, the space $V_h$ is nonconforming, which is not a subspace of $H_0^1(\region)$, but the proof for nonconforming element spaces is the same. 
        \end{remark}

        \begin{Theorem}
            Assume that the mesh size $h$ is sufficiently small and that $\lim\limits_{h\to 0}\norm{u - u_h}_{H^1_h(\region)} = 0$. Then,
            \begin{equation*}
                \norm{u - u_h}_{L^2(\region)}^2 \lesssim \norm{u - u_h}_{H^1_h(\region)}\min_{v_h\in V_h} \norm{\psi - v_h}_{H^1_h(\region)} + \abs{\mE_h(u_h, \nabla \psi)} + \abs{\mE_h(\psi_h, \nabla u)},
            \end{equation*}
            where $\psi\in u^{\perp}\cap H_0^1(\region)$ is the unique solution to the dual problem
            \begin{equation*}
                \dual{(E''(u) - \lambda)\psi, v} = (u - u_h, v)_{L^2(\region)}, \qquad \forall v\in u^{\perp}\cap H_0^1(\region), 
            \end{equation*}
            and $\psi_h\in u^{\perp}\cap V_h$ satisfies
            \begin{equation*}
                (\psi_h, v_h)_{H^1_h(\region)} = (\psi, v_h)_{H^1_h(\region)}, \qquad \forall v_h\in u^{\perp}\cap V_h. 
            \end{equation*}
            Here, $u^{\perp} := \{v\in L^2(\region) : (u, v)_{L^2(\region)} = 0\}$. 
        \end{Theorem}

        \begin{proof}
            Since $(u - u_h, u)_{L^2(\region)} = \frac{1}{2}\norm{u - u_h}_{L^2(\region)}^2$, we denote $\alpha = \frac{1}{2}\norm{u - u_h}_{L^2(\region)}^2$ and $v = (1-\alpha)u - u_h$. For sufficiently small $h$, it holds that
            \begin{equation*}
                \norm{v}_{L^2(\region)} \le \norm{u - u_h}_{L^2(\region)} \lesssim \norm{v}_{L^2(\region)}. 
            \end{equation*}
            By the definition of $\psi$, for any $w \in u^{\perp}\cap H_0^1(\region)$, we have 
            \begin{equation*}
                \dual{(E''(u) - \lambda)\psi, w} = (u - u_h, w)_{L^2(\region)} = (v, w)_{L^2(\region)}. 
            \end{equation*}
            Lemma~\ref{lem:coerc} ensures that for small $h$,
            \begin{equation*}
                \norm{\psi}_{H^1(\region)} \le \mu^{-1}\norm{v}_{L^2(\region)} \le \mu^{-1}\norm{u - u_h}_{L^2(\region)}. 
            \end{equation*}
            Standard elliptic regularity theory implies $\psi\in H^2(\region)$ with $\norm{\psi}_{H^2(\region)} \le \norm{v}_{L^2(\region)}$, and the equation
            \begin{equation*}
                \left(- \Delta + V + 3\beta \abs{u}^2 - \lambda\right) \psi - 2\beta (\abs{u}^2 u, \psi)_{L^2(\region)}u = v
            \end{equation*}
            holds almost everywhere. Consequently, we have 
            \begin{equation}\label{eq:priorL2:1}
                \begin{aligned}
                    \norm{u - u_h}_{L^2(\region)}^2 &\le \norm{v}_{L^2(\region)}^2 = \left( \left(- \Delta + V + 3\beta \abs{u}^2 - \lambda\right) \psi - 2\beta (\abs{u}^2 u, \psi)_{L^2(\region)}u, v \right) \\
                    &= \dual{(E''(u) - \lambda)\psi, v} - \sum_{K\in \mT_h} \int_{\partial K} v(\nabla \psi\cdot n) \diff S\\
                    &= \dual{(E''(u) - \lambda)\psi, u - u_h} - \alpha \dual{(E''(u) - \lambda)\psi, u} - \mE_h(v, \nabla \psi) \\
                    &= \dual{(E''(u) - \lambda)(u - u_h), \psi} + \mE_h(u_h, \nabla \psi) \\
                    &= \dual{(E''(u) - \lambda)(u - u_h), \psi_h} + \dual{(E''(u) - \lambda)(u - u_h), \psi - \psi_h} + \mE_h(u_h, \nabla \psi). 
                \end{aligned}
            \end{equation}
            Given $\psi\in u^{\perp}$ and following the reasoning in \cite{cances2010numerical}*{Theorem 1}, we obtain
            \begin{align*}
                &\norm{\psi_h}_{H^1_h(\region)} \le \norm{\psi}_{H^1(\region)} \le \mu^{-1}\norm{u - u_h}_{L^2(\region)}, \\
                &\norm{\psi - \psi_h}_{H^1_h(\region)} = \min_{v_h\in u^{\perp}\cap V_h}\norm{\psi - v_h}_{H^1_h(\region)} \lesssim \min_{v_h\in V_h}\norm{\psi - v_h}_{H^1_h(\region)}. 
            \end{align*}
            For the first term on the right side of \eqref{eq:priorL2:1}, \eqref{eq:priorH1:3} implies
            \begin{equation}\label{eq:priorL2:2}
                \begin{aligned}
                    \dual{(E''(u) - \lambda)(u - u_h), \psi_h} =& (\lambda_h - \lambda)(u_h, \psi_h)_{L^2(\region)} - \mE_h(\psi_h, \nabla u) \\
                    &\quad - \int_{\region}\beta (u - u_h)^2 (2u + u_h) \psi_h \diff r. 
                \end{aligned}
            \end{equation} 
            Since $\psi_h\in u^{\perp}$, we have
            \begin{equation}\label{eq:priorL2:3}
                \abs{(u_h, \psi_h)_{L^2(\region)}} = \abs{(u - u_h, \psi_h)_{L^2(\region)}} \le \norm{u - u_h}_{L^2(\region)}\norm{\psi_h}_{L^2(\region)} \le \mu^{-1}\norm{u - u_h}_{L^2(\region)}^2. 
            \end{equation}
            Furthermore, applying the H\"older inequality and interpolation estimates, for sufficiently small $h$ yields
            \begin{equation}\label{eq:priorL2:4}
                \begin{aligned}
                    \abs{\int_{\region}\beta (u - u_h)^2 (2u + u_h) \psi_h \diff r} &\le \beta\norm{u - u_h}_{L^3(\region)}^2\norm{2u + u_h}_{H^1_h(\region)}\norm{\psi_h}_{H^1_h(\region)} \\
                    &\lesssim \norm{u - u_h}_{L^2(\region)}^2\norm{u - u_h}_{L^6(\region)} \\
                    &\lesssim \norm{u - u_h}_{L^2(\region)}^2\norm{u - u_h}_{H^1_h(\region)}. 
                \end{aligned}
            \end{equation}
            where the final step utilizes Lemma~\ref{lem:itp}. Combining \eqref{eq:priorL2:3} and \eqref{eq:priorL2:4} with \eqref{eq:priorL2:2} results in
            \begin{equation}\label{eq:priorL2:5}
                \abs{\dual{(E''(u) - \lambda)(u - u_h), \psi_h}} \lesssim \left(\abs{\lambda_h - \lambda} + \norm{u - u_h}_{H^1_h(\region)}\right)\norm{u - u_h}_{L^2(\region)}^2. 
            \end{equation}
            For the second term in \eqref{eq:priorL2:1}, Lemma~\ref{lem:coerc} gives
            \begin{equation}\label{eq:priorL2:6}
                \begin{aligned}
                    \abs{\dual{(E''(u) - \lambda)(u - u_h), \psi - \psi_h}} &\le M\norm{u - u_h}_{H^1_h(\region)}\norm{\psi - \psi_h}_{H^1_h(\region)} \\
                    &\lesssim \norm{u - u_h}_{H^1_h(\region)}\min_{v_h\in V_h} \norm{\psi - v_h}_{H^1_h(\region)}. 
                \end{aligned}
            \end{equation}
            Finally, substituting \eqref{eq:priorL2:5} and \eqref{eq:priorL2:6} into \eqref{eq:priorL2:1} and appealing to Theorem~\ref{thm:asymp}, we conclude that
            \begin{equation*}
                \norm{u - u_h}_{L^2(\region)}^2 \lesssim \norm{u - u_h}_{H^1_h(\region)}\min_{v_h\in V_h} \norm{\psi - v_h}_{H^1_h(\region)} + \abs{\mE_h(u_h, \nabla \psi)} + \abs{\mE_h(\psi_h, \nabla u)}. 
            \end{equation*}
            This completes the proof.
        \end{proof}

        Having established the abstract approximation analysis for the constrained optimization problem~\eqref{eq:gp} in general nonconforming finite element spaces, we proceed in the following section to apply these results to a specific nonconforming finite element space and derive explicit convergence rates with respect to the mesh size $h$.

    \section{Enriched rotated \texorpdfstring{$Q_1$}{Q1} element and lower bound energy approximations}\label{sec:lower}

        Lin proposed the enriched rotated $Q_1$ element in \cite{lin2005superconvergence} and proved its superconvergence properties. A significant feature of this nonconforming finite element method is its ability to provide lower bound estimates for the exact eigenvalues in linear elliptic eigenvalue problems~\cites{li2008lower,hu2014lower}. In this section, we apply our abstract framework to the enriched rotated $Q_1$ element and extend the lower bound property from linear problems to the energy minimum of the nonlinear eigenvalue problem~\eqref{eq:gpe}. We consider the computational domain $\region$ to be a square or a cube, equipped with a regular rectangle or cubic mesh $\mT_h$. Hereafter, we refer to the enriched rotated $Q_1$ element as the $EQ_1^{\mathrm{rot}}$ element.

        \begin{defi}
            On the reference element $\hat{K} = [-1,1]^d$, the function space for the enriched rotated $Q_1$ element is defined as 
            \begin{equation*}
                EQ_1^{\mathrm{rot}}(\hat{K}) = P_1(\hat{K}) + \mathrm{span}\{x_i^2: 1\le i\le d\}, 
            \end{equation*}
            A function $v_h\in L^2(\region)$ is said to belong to the $EQ_1^{\mathrm{rot}}$ finite element space if it satisfies
            \begin{align*}
                &v_h|_K\in EQ_1^{\mathrm{rot}}(\hat{K}), \qquad \forall K\in \mT_h, \\
                &\int_e [v_h] \diff S := \int_e v_h \diff S = 0, \qquad \forall e\in \mF_h, e\subseteq \partial\region, \\
                &\int_e [v_h] \diff S := \int_e (v_h|_{K_1} - v_h|_{K_2}) \diff S = 0, \qquad \forall K_1\cap K_2 = e\in \mF_h. 
            \end{align*}
        \end{defi}

        We first establish the a priori error estimates for the $EQ_1^{\mathrm{rot}}$ element.

        \begin{Theorem}\label{thm:prior}
            For the $EQ_1^{\mathrm{rot}}$ element, the following a priori error estimates hold for sufficiently small mesh sizes $h$:
            \begin{align*}
                &\abs{\lambda - \lambda_h} \lesssim h^2, \qquad \abs{E(u) - E(u_h)} \lesssim h^2, \\
                &\norm{u - u_h}_{L^2(\region)} + h\norm{u - u_h}_{H^1_h(\region)} \lesssim h^2. 
            \end{align*}
        \end{Theorem}

        \begin{proof}
            First, by invoking the approximation properties and consistency error estimates, we have
            \begin{equation*}
                \norm{u - u_h}_{H^1_h(\region)} \lesssim \inf_{v_h\in V_h} \norm{u - v_h}_{H^1_h(\region)} + \sup_{\substack{w_h \in V_h \\ \norm{w_h}_{H^1_h(\region)}\le 1}} \abs{\mE_h\left(w_h, \nabla u\right)} \lesssim h. 
            \end{equation*}
            Next, considering the consistency errors that appear in the estimates for the eigenvalues, energy, and $L^2$-norm, we apply the obtained $H^1_h$-error bounds to find
            \begin{align*}
                &\abs{\mE_h(u_h, \nabla u)} = \abs{\mE_h(u- u_h, \nabla u)} \lesssim h\norm{u}_{H^2(\region)}\norm{u - u_h}_{H^1_h(\region)} \lesssim h^2, \\
                &\abs{\mE_h(u_h, \nabla \psi)} = \abs{\mE_h(u- u_h, \nabla \psi)} \lesssim h\norm{\psi}_{H^2(\region)}\norm{u - u_h}_{H^1_h(\region)} \lesssim h^2\norm{u - u_h}_{L^2(\region)}, \\
                &\abs{\mE_h(\psi_h, \nabla u)} = \abs{\mE_h(\psi - \psi_h, \nabla u)} \lesssim h\norm{u}_{H^2(\region)}\norm{\psi - \psi_h}_{H^1_h(\region)} \lesssim h^2\norm{\psi}_{H^2(\region)} \lesssim h^2\norm{u - u_h}_{L^2(\region)}. 
            \end{align*}
            Thus, for the $L^2$-norm error, it holds that
            \begin{equation*}
                \begin{aligned}
                    \norm{u - u_h}_{L^2(\region)}^2 &\lesssim \norm{u - u_h}_{H^1_h(\region)}\min_{v_h\in V_h} \norm{\psi - v_h}_{H^1_h(\region)} + \abs{\mE_h(u_h, \nabla \psi)} + \abs{\mE_h(\psi_h, \nabla u)} \\
                    &\lesssim h^2\left( \norm{\psi}_{H^2(\region)} + \norm{u - u_h}_{L^2(\region)} \right) \lesssim h^2\norm{u - u_h}_{L^2(\region)}, 
                \end{aligned}
            \end{equation*}
            which implies $\norm{u - u_h}_{L^2(\region)} \lesssim h^2$. Finally, combining the $H^1_h$, $L^2$, and consistency error estimates, we obtain
            \begin{align*}
                &\abs{\lambda - \lambda_h} \lesssim \norm{u-u_h}_{H^1_h(\region)}^2 + \norm{u-u_h}_{L^2(\region)} + \abs{\mE_h(u_h, \nabla u)} \lesssim h^2, \\
                &\abs{E(u) - E(u_h)} \lesssim \norm{u - u_h}_{H^1_h(\region)}^2 + \abs{\mE_h(u_h, \nabla u)} \lesssim h^2. 
            \end{align*}
            The proof is complete. 
        \end{proof}

        Having established the a priori estimates, we now address the lower bound energy approximation property of the $EQ_1^{\mathrm{rot}}$ element. It is crucial to emphasize that verifying the saturation condition is an essential and key step in proving this lower bound property. This conclusion has been previously proven by Jun Hu, Yunqing Huang and Qun Lin~\cite{hu2014lower} in the context of linear eigenvalue problems. Motivated by this, we first provide the saturation condition for the $EQ_1^{\mathrm{rot}}$ finite element space.
        
        \begin{lemma}\label{lem:saturation}
            When $d = 2$, the ground state $u$ and the space $V_h$ satisfies 
            \begin{align}
                \label{eq:saturation:1} &\norm{\frac{\partial^2 u}{\partial x \partial y}}_{L^2(\region)} \neq 0, \\
                \label{eq:saturation:4} &\norm{\frac{\partial^2 v_h}{\partial x \partial y}}_{L^2(\region)} = 0, \qquad \forall v_h \in V_h. 
            \end{align}
            When $d = 3$, the ground state $u$ and the space $V_h$ satisfies 
            \begin{align}
                \label{eq:saturation:2} &\norm{\frac{\partial^2 u}{\partial x \partial y}}_{L^2(\region)} + \norm{\frac{\partial^2 u}{\partial y \partial z}}_{L^2(\region)} + \norm{\frac{\partial^2 u}{\partial z \partial x}}_{L^2(\region)} \neq 0, \\
                \label{eq:saturation:5} &\norm{\frac{\partial^2 v_h}{\partial x \partial y}}_{L^2(\region)} = \norm{\frac{\partial^2 v_h}{\partial y \partial z}}_{L^2(\region)} = \norm{\frac{\partial^2 v_h}{\partial z \partial x}}_{L^2(\region)} = 0, \qquad \forall v_h \in V_h. 
            \end{align}
            Further, there holds the saturation condition 
            \begin{equation}\label{eq:saturation:3}
                h \lesssim \norm{\nabla_h (u - u_h)}_{L^2(\region)}. 
            \end{equation}
        \end{lemma}

        \begin{proof}
            For simplicity, we only prove the case where $d=2$, and we proceed by contradiction. Assume that $\region$ is a square $[x_{\min}, x_{\max}] \times [y_{\min}, y_{\max}]$, and $\norm{\frac{\partial^2 u}{\partial x \partial y}}_{L^2(\region)} = 0$, then there exist functions $f(x) \in H^1([x_{\min}, x_{\max}])$ and $g(y) \in H^1([y_{\min}, y_{\max}])$ such that 
            \begin{equation*}
                u(x,y) = f(x) + g(y). 
            \end{equation*}
            Noting that $u \in H_0^1(\region)$. In line $\{x_{\min}\} \times [y_{\min}, y_{\max}]$, we have 
            \begin{equation*}
                0 = u(x, y) = u(x_{\min}, y) = f(x_{\min}) + g(y), 
            \end{equation*}
            thus $g(y) = - f(x_{\min})$ is a constant. Similarly, $f(x)$ is a constant, therefore $u \equiv 0$ in $\region$, which contradicts the fact that $\norm{u}_{L^2(\region)} = 1$. So~\eqref{eq:saturation:1} is true. \eqref{eq:saturation:4} can be directly obtained according to the definition of $EQ_1^{\mathrm{rot}}$ space. Finally, According to~\cite{hu2014lower}*{Theorem 8.1}, the saturation condition~\eqref{eq:saturation:3} holds. 
        \end{proof}
        
        For following analysis, we specifically require the trapping potential $V \in H^1(\region)$. We define the interpolation operator~\cite{hu2014lower} $\Pi_h: H_0^1(\region) \to V_h$ associated with the $EQ_1^{\mathrm{rot}}$ element, which satisfies
        \begin{align*}
            &\int_e \Pi_h v \diff S = \int_e v \diff S, \qquad \forall v\in H_0^1(\region), e\in \mF_h, \\
            &\int_K \Pi_h v \diff r = \int_K v \diff r, \qquad \forall v\in H_0^1(\region), K\in \mT_h. 
        \end{align*}
        According to \cite{hu2014lower}*{Lemma 5.5}, this interpolation operator possesses the following orthogonality property
        \begin{equation}\label{eq:orth}
            \left( \nabla_h (u - \Pi_h u), \nabla_h v_h \right)_{L^2(\region)} = 0, \qquad \forall v_h\in V_h. 
        \end{equation}
        Furthermore, the approximation property of the interpolation operator \cite{li2008lower}*{Lemma 1} ensures
        \begin{equation}\label{eq:itp1}
            \norm{u - \Pi_h u}_{L^2(\region)} + h\norm{u - \Pi_h u}_{H^1_h(\region)} \lesssim h^2. 
        \end{equation}
        Let $\Pi_0$ denote the piecewise constant projection operator onto $EQ_1^{\mathrm{rot}}$, defined such that for any $K\in \mT_h$ and $v\in H_0^1(\region)$, $\Pi_0 v|_{K} = \abs{K}^{-1}\int_K v \diff r$. Applying the Poincar\'e inequality piecewise yields
        \begin{equation}\label{eq:itp0}
            \norm{v - \Pi_0 v}_{L^2(\region)} \lesssim h \norm{\nabla_h v}_{L^2(\region)}, \qquad \forall v\in V_h + H_0^1(\region). 
        \end{equation}
        With these preliminaries, we present the lower bound property of the $EQ_1^{\mathrm{rot}}$ element for the energy minimum in Theorem~\ref{thm:lowerEnergy}.

        \begin{Theorem}\label{thm:lowerEnergy}
            Let $E(u)$ and $E(u_h)$ denote the exact and discrete ground state energies, respectively. For a sufficiently small mesh size $h$, it holds that
            \begin{equation*}
                E(u_h) \le E(u). 
            \end{equation*}
        \end{Theorem}

        \begin{proof}
            By the definitions of $u$ and $u_h$, we have
            \begin{equation}\label{eq:lowerEnergy:1}
                \begin{aligned}
                    E(u) - E(u_h) &= \frac{1}{2}\dual{A_u u, u} - \frac{1}{2}\dual{A_{u_h} u_h, u_h} - \frac{1}{4}\int_{\region} \beta \left( \abs{u}^4 - \abs{u_h}^4 \right) \diff r\\
                    &= \frac{1}{2}\dual{A_{u_h} u, u} - \frac{1}{2}\dual{A_{u_h} u_h, u_h} + \frac{1}{4}\int_{\region} \beta \left( \abs{u}^2 - \abs{u_h}^2 \right)^2 \diff r\\
                    &= \frac{1}{2}\dual{A_{u_h} (u - u_h), u - u_h} + \dual{A_{u_h}(u - u_h), u_h} + \frac{1}{4}\int_{\region} \beta \left( \abs{u}^2 - \abs{u_h}^2 \right)^2 \diff r. 
                \end{aligned}
            \end{equation}
            Regarding the first term in \eqref{eq:lowerEnergy:1}, according to \eqref{eq:saturation:3} in Lemma~\ref{lem:saturation}, there exists a constant $C_1 > 0$ such that 
            \begin{equation}\label{eq:lowerEnergy:11}
                \frac{1}{2}\dual{A_{u_h} (u - u_h), u - u_h} \ge \frac{1}{2}\norm{\nabla_h (u - u_h)}_{L^2(\region)}^2 \ge C_1 h^2. 
            \end{equation}
            For the second term in \eqref{eq:lowerEnergy:1}, we can write
            \begin{equation}\label{eq:lowerEnergy:2}
                \begin{aligned}
                    \dual{A_{u_h}(u - u_h), u_h} =& \dual{A_{u_h}u, u_h} - \dual{A_{u_h} u_h, u_h} \\
                    =& \dual{A_{u_h}\left( u - \Pi_h u \right), u_h} + \dual{A_{u_h} u_h, \left( \Pi_h u - u_h \right)} \\
                    =& \dual{A_{u_h}\left( u - \Pi_h u \right), u_h} + \lambda_h \left( \left(u_h, \Pi_h u\right) - 1 \right) \\
                    =& \dual{A_{u_h}\left( u - \Pi_h u \right), u_h} + \frac{\lambda_h}{2} \left( \norm{\Pi_h u}_{L^2(\region)}^2 - \norm{u}_{L^2(\region)}^2 \right) \\
                    &\quad - \frac{\lambda_h}{2} \norm{u_h - \Pi_h u}_{L^2(\region)}^2. 
                \end{aligned}
            \end{equation}
            Expanding the first term of \eqref{eq:lowerEnergy:2}, we have
            \begin{equation}\label{eq:lowerEnergy:3}
                \begin{aligned}
                    \dual{A_{u_h}\left( u - \Pi_h u \right), u_h} =& \left(\left(V + \beta \abs{u_h}^2\right)\left( u - \Pi_h u \right), u_h\right)_{L^2(\region)} \\
                    =& \left(\beta \left(\abs{u_h}^2 - \abs{u}^2\right)\left( u - \Pi_h u \right), u_h\right)_{L^2(\region)}\\
                    &\quad + \left(\left(V + \beta \abs{u}^2 - \Pi_0 \left(V + \beta \abs{u}^2\right)\right)\left( u - \Pi_h u \right), u_h\right)_{L^2(\region)} \\
                    &\quad + \left(\Pi_0 \left(V + \beta \abs{u}^2\right)\left( u - \Pi_h u \right), u_h\right)_{L^2(\region)}. 
                \end{aligned}
            \end{equation}
            For the first term on the right-hand side of \eqref{eq:lowerEnergy:3}, we obtain
            \begin{equation}\label{eq:lowerEnergy:13}
                \begin{aligned}
                    &\abs{\left(\beta \left(\abs{u_h}^2 - \abs{u}^2\right)\left( u - \Pi_h u \right), u_h\right)_{L^2(\region)}} \\
                    =& \abs{\int_{\region} \beta \left( u_h - u \right)\left( u_h + u \right)\left( u - \Pi_h u \right)u_h \diff r} \\
                    \le& \beta \norm{u_h - u}_{L^3(\region)}\norm{u_h + u}_{L^6(\region)}\norm{u - \Pi_h u}_{L^3(\region)}\norm{u_h}_{L^6(\region)}. 
                \end{aligned}
            \end{equation}
            Based on the interpolation inequality and the a priori estimation Theorem~\ref{thm:prior}, we have
            \begin{equation}\label{eq:lowerEnergy:14}
                \norm{u_h - u}_{L^3(\region)} \lesssim \norm{u_h - u}_{L^2(\region)}^{\frac{1}{2}}\norm{u_h - u}_{L^6(\region)}^{\frac{1}{2}} \lesssim \norm{u_h - u}_{L^2(\region)}^{\frac{1}{2}}\norm{u_h - u}_{H^1_h(\region)}^{\frac{1}{2}} \lesssim h^{\frac{3}{2}}. 
            \end{equation}
            Similarly, based on the interpolation inequality and the properties of the interpolation operator $\Pi_h$ in \eqref{eq:itp1}, we have
            \begin{equation}\label{eq:lowerEnergy:15}
                \norm{u - \Pi_h u}_{L^3(\region)} \lesssim h^{\frac{3}{2}}. 
            \end{equation}
            Substituting \eqref{eq:lowerEnergy:14} and \eqref{eq:lowerEnergy:15} into \eqref{eq:lowerEnergy:13}, we obtain the estimate for the first term on the right-hand side of \eqref{eq:lowerEnergy:3}:
            \begin{equation}\label{eq:lowerEnergy:16}
                \abs{\left(\beta \left(\abs{u_h}^2 - \abs{u}^2\right)\left( u - \Pi_h u \right), u_h\right)_{L^2(\region)}} \lesssim h^3. 
            \end{equation}
            For the second term on the right-hand side of \eqref{eq:lowerEnergy:3}, we have
            \begin{equation}\label{eq:lowerEnergy:17}
                \begin{aligned}
                    &\abs{\left(\left(V + \beta \abs{u}^2 - \Pi_0 \left(V + \beta \abs{u}^2\right)\right)\left( u - \Pi_h u \right), u_h\right)_{L^2(\region)}} \\
                    =& \abs{\int_{\region} \left(V + \beta \abs{u}^2 - \Pi_0 \left(V + \beta \abs{u}^2\right)\right)\left( u - \Pi_h u \right)u_h \diff r} \\
                    \le& \norm{V + \beta \abs{u}^2 - \Pi_0 \left(V + \beta \abs{u}^2\right)}_{L^2(\region)}\norm{u - \Pi_h u}_{L^3(\region)}\norm{u_h}_{L^6(\region)} 
                \end{aligned}
            \end{equation}
            Noting that 
            \begin{equation*}
                \norm{\nabla \left(V + \beta \abs{u}^2\right)}_{L^2(\region)} \le \norm{V}_{H^1(\region)} + \beta\norm{u}_{L^{\infty}(\region)}\norm{u}_{H^1(\region)}
            \end{equation*}
            is bounded, and according to \eqref{eq:itp0}, we have
            \begin{equation}\label{eq:lowerEnergy:18}
                \norm{V + \beta \abs{u}^2 - \Pi_0 \left(V + \beta \abs{u}^2\right)}_{L^2(\region)} \lesssim h\norm{\nabla \left(V + \beta \abs{u}^2\right)}_{L^2(\region)} \lesssim h. 
            \end{equation}
            Substituting \eqref{eq:lowerEnergy:15} and \eqref{eq:lowerEnergy:18} into \eqref{eq:lowerEnergy:17}, we obtain the estimate for the second term on the right-hand side of \eqref{eq:lowerEnergy:3}:
            \begin{equation}\label{eq:lowerEnergy:4}
                \abs{\left(\left(V + \beta \abs{u}^2 - \Pi_0 \left(V + \beta \abs{u}^2\right)\right)\left( u - \Pi_h u \right), u_h\right)_{L^2(\region)}} \lesssim h^{\frac{5}{2}}. 
            \end{equation}
            For the third term on the right-hand side of \eqref{eq:lowerEnergy:3}, we note that for any $K\in \mT_h$, $\int_{K} \left(u - \Pi_h u\right) \diff r = 0$, hence
            \begin{equation}\label{eq:lowerEnergy:5}
                \begin{aligned}
                    &\abs{\left(\Pi_0 \left(V + \beta \abs{u}^2\right)\left( u - \Pi_h u \right), u_h\right)_{L^2(\region)}} \\
                    =& \abs{\left(\Pi_0 \left(V + \beta \abs{u}^2\right)\left( u - \Pi_h u \right), u_h - \Pi_0 u_h\right)_{L^2(\region)}} \\
                    \le& \norm{\Pi_0 \left(V + \beta \abs{u}^2\right)}_{L^{\infty}(\region)}\norm{\left( u - \Pi_h u \right)}_{L^2(\region)} \norm{u_h - \Pi_0 u_h}_{L^2(\region)} \\
                    \lesssim& \norm{V + \beta \abs{u}^2}_{L^{\infty}(\region)}\left( h^2\norm{u}_{H^2(\region)} \right)\left( h\norm{\nabla_h u_h}_{L^2(\region)} \right) \\
                    \le& h^3 \left(\norm{V}_{L^{\infty}(\region)} + \beta\norm{u}_{L^{\infty}(\region)}^2\right)\norm{u}_{H^2(\region)}\norm{u_h}_{H^1_h(\region)} \lesssim h^3. 
                \end{aligned}
            \end{equation}
            Substituting \eqref{eq:lowerEnergy:16}, \eqref{eq:lowerEnergy:4}, and \eqref{eq:lowerEnergy:5} into \eqref{eq:lowerEnergy:3}, we obtain the estimate for the first term on the right-hand side of \eqref{eq:lowerEnergy:2}:
            \begin{equation}\label{eq:lowerEnergy:6}
                \abs{\dual{A_{u_h}\left( u - \Pi_h u \right), u_h}} \lesssim h^{\frac{5}{2}} + h^3. 
            \end{equation}
            Regarding the second term on the right-hand side of \eqref{eq:lowerEnergy:2}, we have
            \begin{equation}\label{eq:lowerEnergy:7}
                \begin{aligned}
                    \abs{\norm{\Pi_h u}_{L^2(\region)}^2 - \norm{u}_{L^2(\region)}^2} &= \abs{\left( \Pi_h u - u, \Pi_h u + u \right)_{L^2(\region)}} \\
                    &= \abs{\left( \Pi_h u - u, \Pi_h u + u - \Pi_0\left( \Pi_h u + u \right)\right)_{L^2(\region)}} \\
                    &\le \norm{\Pi_h u - u}_{L^2(\region)}\norm{\Pi_h u + u - \Pi_0\left( \Pi_h u + u \right)}_{L^2(\region)} \\
                    &\lesssim \left( h^2 \norm{u}_{H^2(\region)} \right) \left( h\norm{\nabla_h \left(\Pi_h u + u\right)}_{L^2(\region)} \right) \\
                    &\le 2 h^3 \norm{u}_{H^2(\region)} \norm{u}_{H^1(\region)} \lesssim h^3. 
                \end{aligned}
            \end{equation}
            For the third term on the right-hand side of \eqref{eq:lowerEnergy:2}, we have
            \begin{equation}\label{eq:lowerEnergy:8}
                \norm{u_h - \Pi_h u}_{L^2(\region)}^2 \le 2 \left( \norm{u - u_h}_{L^2(\region)}^2 + \norm{u - \Pi_h u}_{L^2(\region)}^2 \right) \lesssim h^4. 
            \end{equation}
            Substituting \eqref{eq:lowerEnergy:6}--\eqref{eq:lowerEnergy:8} into \eqref{eq:lowerEnergy:2}, we obtain
            \begin{equation}\label{eq:lowerEnergy:9}
                \abs{\dual{A_{u_h}(u - u_h), u_h}} \le h^{\frac{5}{2}} + h^3 + h^4. 
            \end{equation}
            For the third term on the right-hand side of \eqref{eq:lowerEnergy:1}, we have
            \begin{equation}\label{eq:lowerEnergy:10}
                \begin{aligned}
                    \frac{1}{4}\int_{\region} \beta \left( \abs{u}^2 - \abs{u_h}^2 \right)^2 \diff r &\le \frac{\beta}{4}\norm{u - u_h}_{L^3(\region)}^2\norm{u + u_h}_{L^6(\region)}^2 \\
                    &\lesssim \norm{u - u_h}_{L^2(\region)}\norm{u - u_h}_{L^6(\region)}\norm{u + u_h}_{H^1_h(\region)}^2 \\
                    &\lesssim \norm{u - u_h}_{L^2(\region)}\norm{u - u_h}_{H^1_h(\region)}\norm{u + u_h}_{H^1_h(\region)}^2 \lesssim h^3, 
                \end{aligned}
            \end{equation}
            where the second inequality comes from the interpolation inequality, the third inequality comes from Lemma~\ref{lem:itp}, and the fourth comes from the a priori estimation Theorem~\ref{thm:prior}. Combining \eqref{eq:lowerEnergy:9} and \eqref{eq:lowerEnergy:10}, there exists a constant $C_2 > 0$ such that
            \begin{equation}\label{eq:lowerEnergy:12}
                \abs{\dual{A_{u_h}(u - u_h), u_h}} + \frac{1}{4}\int_{\region} \beta \left( \abs{u}^2 - \abs{u_h}^2 \right)^2 \diff r \le C_2 \left( h^{\frac{5}{2}} + h^3 + h^4 \right). 
            \end{equation}
            Substituting \eqref{eq:lowerEnergy:11} and \eqref{eq:lowerEnergy:12} into \eqref{eq:lowerEnergy:1}, we finally obtain
            \begin{equation}\label{eq:lowerEnergy:19}
                \begin{aligned}
                    E(u) - E(u_h) &\ge \frac{1}{2}\dual{A_{u_h} (u - u_h), u - u_h} - \abs{\dual{A_{u_h}(u - u_h), u_h}} - \frac{1}{4}\int_{\region} \beta \left( \abs{u}^2 - \abs{u_h}^2 \right)^2 \diff r \\
                    &\ge C_1 h^2 - C_2 \left( h^{\frac{5}{2}} + h^3 + h^4 \right) = C_2 h^2 \left( \frac{C_1}{C_2} - h^{\frac{1}{2}} - h - h^2 \right). 
                \end{aligned}
            \end{equation}
            When $h$ is sufficiently small such that $h^{\frac{1}{2}} + h + h^2 \le \frac{C_1}{C_2}$, we have $E(u) \ge E(u_h)$, meaning that the $EQ_1^{\mathrm{rot}}$ element provides a lower bound estimate for the energy. 
        \end{proof}

        \begin{remark}
            According to the Courant--Fischer min-max theorem, the eigenvalues and energies coincide for linear eigenvalue problems, making the lower bound property naturally applicable to eigenvalues. For the nonlinear Gross--Pitaevskii eigenvalue problem~\eqref{eq:gpe}, a natural question arises as to whether a similar lower bound property holds for the eigenvalues themselves. Numerical experiments presented in Section~\ref{sec:numerical} suggest that this property holds at least in certain cases.
        \end{remark}

        Combining Theorem~\ref{thm:prior} and \ref{thm:lowerEnergy}, we can obtain the monotonicity of energy to a certain extent. 

        \begin{cor}\label{coro:monotonicity}
            For a sufficiently small mesh size $H$, there exists a constant $\eta < 1$, such that for all mesh size $h \le \eta H$, we have 
            \begin{equation*}
                E(u_H) \le E(u_h) \le E(u). 
            \end{equation*}
        \end{cor}

        \begin{proof}
            According to Theorem~\ref{thm:prior} and \ref{thm:lowerEnergy}, there exists a constant $C_1$ such that 
            \begin{equation*}
                E(u) - E(u_h) \le C_1 h^2. 
            \end{equation*}
            According to \eqref{eq:lowerEnergy:19} in Theorem~\ref{thm:lowerEnergy}, there exist constants $C_1$ and $C_2$ such that 
            \begin{equation*}
                E(u) - E(u_H) \ge C_2 \left( C_3 - H^{\frac{1}{2}} - H - H^2 \right) H^2. 
            \end{equation*}
            For any $\delta < C_3$, as long as $H$ and $\eta$ satisfy 
            \begin{equation*}
                H^{\frac{1}{2}} + H + H^2 \le C_3 - \delta, \qquad \eta^2 \le \frac{C_2 \delta}{C_1}, 
            \end{equation*}
            we obtain 
            \begin{equation*}
                \begin{aligned}
                    E(u_h) - E(u_H) &= E(u_h) - E(u) + E(u) - E(u_H) \\
                    &\ge - C_1 h^2 + C_2 \left( C_3 - H^{\frac{1}{2}} - H - H^2 \right) H^2 \\
                    &\ge C_1 \left( \eta^2 H^2 - h^2 \right), 
                \end{aligned}
            \end{equation*}
            which leads to the conclusion that $E(u_H) \le E(u_h)$ when $h \le \eta H$. 
        \end{proof}

        This concludes our theoretical analysis. In the following section, we validate Theorems~\ref{thm:prior} and~\ref{thm:lowerEnergy} through numerical experiments.

    \section{Numerical experiments}\label{sec:numerical}

        We conclude this paper with a series of numerical experiments. All computations were implemented in MATLAB and performed on a desktop computer equipped with an Intel Core i5-14600KF processor and 32GB of RAM. In this section, we consider the optimization problem of the Gross--Pitaevskii energy functional~\eqref{eq:gp} on a two-dimensional square domain. Since analytical solutions to \eqref{eq:gp} are generally unavailable, we employ a reference solution computed using the $Q_2$ finite element method on a fine mesh consisting of $2^{10 \times 2}$ elements. To solve the discrete optimization problem~\eqref{eq:gph} in the finite element space, we utilize the Sobolev gradient flow method~\cite{henning2020sobolev} with a fixed step size of $1$. The iterations are terminated once the residual falls below $10^{-12}$. Notably, the Sobolev gradient flow can be interpreted as a preconditioned $L^2$ gradient flow. To further enhance computational efficiency for large-scale simulations, one might incorporate Riemannian acceleration techniques~\cite{shao2025riemannian}, which remains a direction for our future research.

        \subsection{Discrete characteristics and morphological evolution of ground states}\label{subsec:gs}

        First, we examine the morphological evolution of the ground state under varying mesh densities. We consider Case I from \cite{bao2004computing}*{Example 3} as a benchmark, where the computational domain is $\region = [-4,4]\times[-8,8]$, the repulsion parameter is $\beta = 400$, and the trapping potential is an anisotropic harmonic potential $V(x)=16x_1^2+x_2^2$. Figure~\ref{fig:gs} illustrates the transformation of the discrete ground state $u_h$ as the mesh is refined. On coarser grids, the discontinuities of the $EQ_1^{\mathrm{rot}}$ element across element interfaces are clearly visible, providing a direct visual representation of the nonconforming nature of the space. As the mesh size $h$ decreases, the discrete solution gradually converges toward a smooth, exact ground state.

        \begin{figure}[htb]
            \centering
            \begin{minipage}[c]{0.15\textwidth}
                \centering
                \includegraphics[width=1\textwidth]{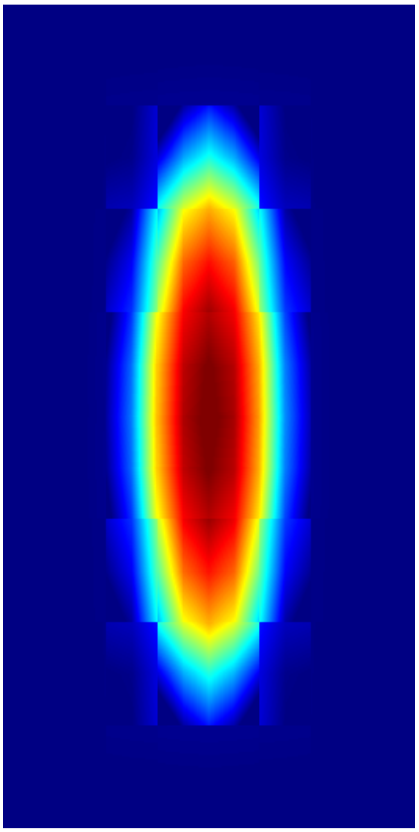}
                \subcaption{$8\times 8$}
            \end{minipage}
            \begin{minipage}[c]{0.15\textwidth}
                \centering
                \includegraphics[width=1\textwidth]{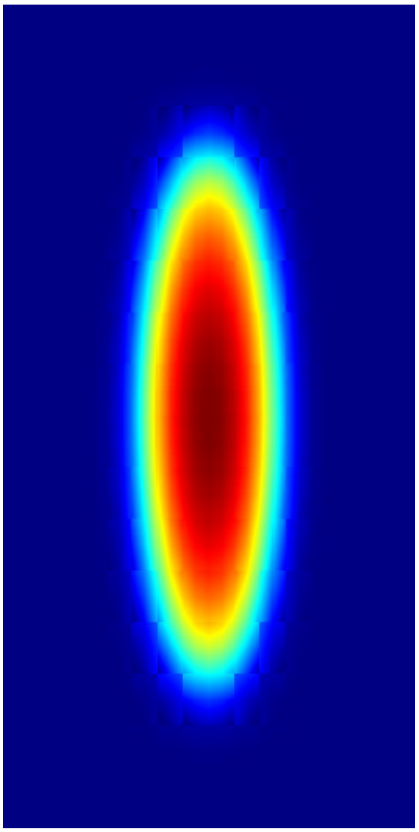}
                \subcaption{$16\times 16$}
            \end{minipage}
            \begin{minipage}[c]{0.15\textwidth}
                \centering
                \includegraphics[width=1\textwidth]{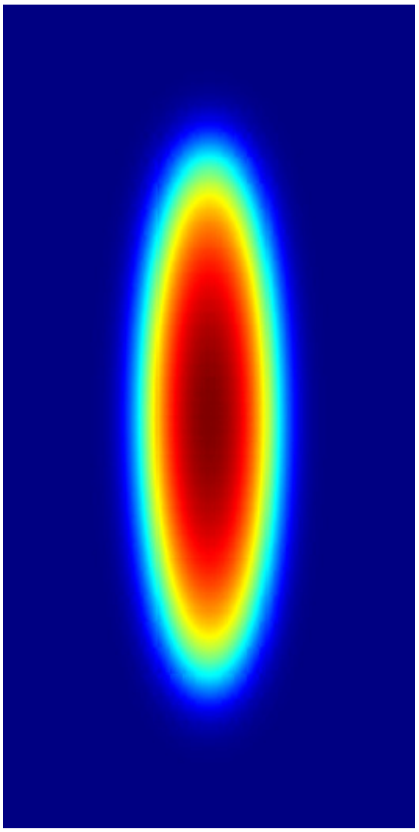}
                \subcaption{$32\times 32$}
            \end{minipage}
            \begin{minipage}[c]{0.15\textwidth}
                \centering
                \includegraphics[width=1\textwidth]{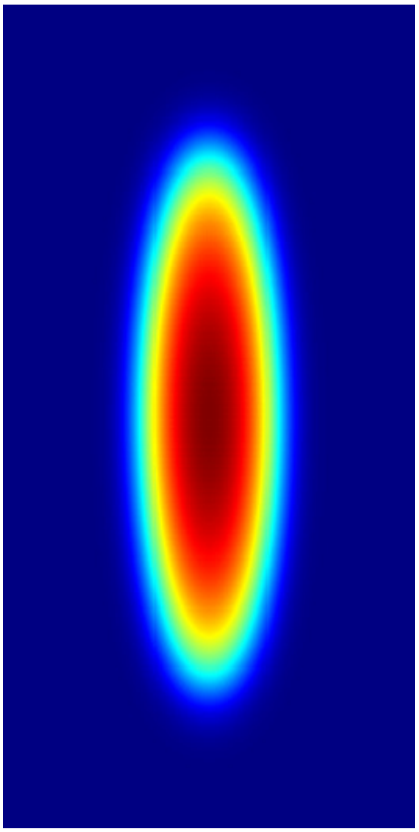}
                \subcaption{$64\times 64$}
            \end{minipage}
            \begin{minipage}[c]{0.15\textwidth}
                \centering
                \includegraphics[width=1\textwidth]{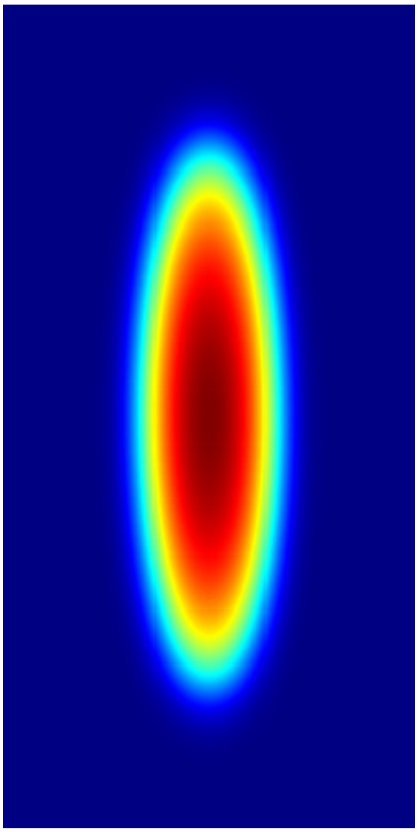}
                \subcaption{$128\times 128$}
            \end{minipage}
            \begin{minipage}[c]{0.15\textwidth}
                \centering
                \includegraphics[width=1\textwidth]{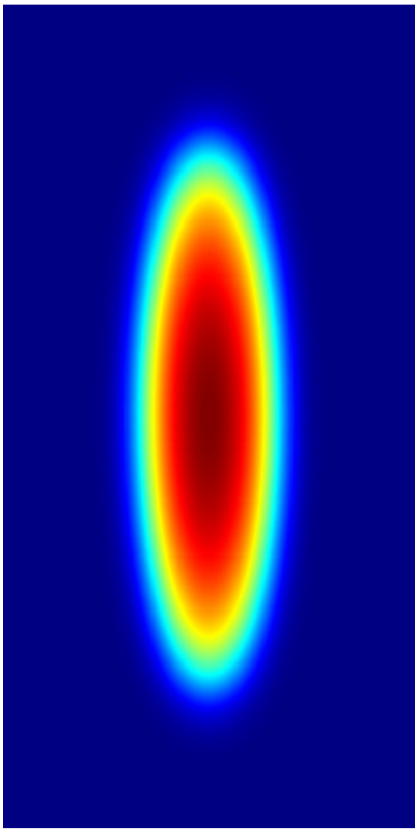}
                \subcaption{$256\times 256$}
            \end{minipage}
            \caption{Evolution of the ground state solution under different mesh sizes.}
            \label{fig:gs}
        \end{figure}

        \subsection{Numerical validation of convergence rates}\label{subsec:tab}

        To quantitatively verify the a priori error estimates established in Theorem~\ref{thm:prior}, we adopt the experimental setup from \cite{chen2024fully}. The domain is set as $\region = [-1,1]^2$, with parameter $\beta = 1$, and the potential function is defined as:
        \begin{equation*}
            V(x) = 1 - \sin^2 \left(\frac{\pi}{2}(x_1+1)\right) \sin^2 \left(\frac{\pi}{2}(x_2+1)\right). 
        \end{equation*}
        We investigate the behavior of the discrete energy $E(u_h)$, the discrete eigenvalue $\lambda_h$, and the solution errors as the number of grid cells ($N^2$) increases. Detailed results are summarized in Table~\ref{tab:tab1} and \ref{tab:tab2}. The data indicate that the energy, eigenvalue, and $L^2$ errors all exhibit a convergence rate of $O(h^2)$, while the $H^1$ error follows an $O(h)$ rate. These numerical results are in perfect agreement with the theoretical predictions of Theorem~\ref{thm:prior}, demonstrating that the $EQ_1^{\mathrm{rot}}$ element performs well for the Gross--Pitaevskii eigenvalue problem~\eqref{eq:gpe}.

        \begin{table}[htb]
            \centering
            \caption{Convergence results of the $L^2$ and $H^1$ errors for the $EQ_1^{\mathrm{rot}}$ element in Example~\ref{subsec:tab}.}
            \label{tab:tab1}

            \begin{tabular}{lll ll ll}
                \toprule
                \multirow{2}{*}{$N$} & \multirow{2}{*}{DOFs} & \multirow{2}{*}{cpu(s)} & 
                \multicolumn{2}{l}{$L^2$ error} & 
                \multicolumn{2}{l}{$H^1$ error} \\
                \cmidrule(r){4-5} \cmidrule(l){6-7}
                & & & error & order & error & order \\
                \midrule
                8    & 208     & 0.02   & $1.28\times 10^{-2}$ & 1.98 & $2.52\times 10^{-1}$ & 1.00 \\
                16   & 800     & 0.04   & $3.21\times 10^{-3}$ & 2.00 & $1.26\times 10^{-1}$ & 1.00 \\
                32   & 3136    & 0.11   & $8.03\times 10^{-4}$ & 2.00 & $6.30\times 10^{-2}$ & 1.00 \\
                64   & 12416   & 0.42   & $2.01\times 10^{-4}$ & 2.00 & $3.15\times 10^{-2}$ & 1.00 \\
                128  & 49408   & 1.88   & $5.02\times 10^{-5}$ & 2.00 & $1.57\times 10^{-2}$ & 1.00 \\
                256  & 197120  & 8.50   & $1.25\times 10^{-5}$ & 2.00 & $7.87\times 10^{-3}$ & 1.00 \\
                512  & 787456  & 33.74  & $3.14\times 10^{-6}$ & 2.00 & $3.93\times 10^{-3}$ & 1.00 \\
                1024 & 3147776 & 137.32 & $7.84\times 10^{-7}$ & 2.00 & $1.97\times 10^{-3}$ & 1.00 \\
                \bottomrule
            \end{tabular}
        \end{table}

        \begin{table}[htb]
            \centering
            \caption{Convergence results of the energy and eigenvalue for the $EQ_1^{\mathrm{rot}}$ element in Example~\ref{subsec:tab}.}
            \label{tab:tab2}

            \begin{tabular}{l lll lll}
                \toprule
                \multirow{2}{*}{$N$} & 
                \multicolumn{3}{l}{energy ($E$)} & 
                \multicolumn{3}{l}{eigenvalue ($\lambda$)} \\
                \cmidrule(r){2-4} \cmidrule(l){5-7}
                & value & error & order & value & error & order \\
                \midrule
                8    & 2.795872 & $3.09\times 10^{-2}$ & 1.89 & 5.872934 & $6.19\times 10^{-2}$ & 1.90 \\
                16   & 2.818900 & $7.88\times 10^{-3}$ & 1.97 & 5.919046 & $1.58\times 10^{-2}$ & 1.97 \\
                32   & 2.824798 & $1.98\times 10^{-3}$ & 1.99 & 5.930845 & $3.96\times 10^{-3}$ & 1.99 \\
                64   & 2.826281 & $4.95\times 10^{-4}$ & 2.00 & 5.933812 & $9.91\times 10^{-4}$ & 2.00 \\
                128  & 2.826652 & $1.24\times 10^{-4}$ & 2.00 & 5.934555 & $2.48\times 10^{-4}$ & 2.00 \\
                256  & 2.826745 & $3.10\times 10^{-5}$ & 2.00 & 5.934740 & $6.19\times 10^{-5}$ & 2.00 \\
                512  & 2.826768 & $7.74\times 10^{-6}$ & 2.00 & 5.934787 & $1.55\times 10^{-5}$ & 2.00 \\
                1024 & 2.826774 & $1.93\times 10^{-6}$ & 2.00 & 5.934798 & $3.87\times 10^{-6}$ & 2.00 \\
                \bottomrule
            \end{tabular}
        \end{table}

        \subsection{Comparison of conforming and nonconforming elements and energy lower bound property}\label{subsec:fig1}

        To highlight the unique properties of the $EQ_1^{\mathrm{rot}}$ element, we compare it with the classical conforming $Q_2$ finite element. The parameters are identical to those in Section~\ref{subsec:tab}, but we use a stronger repulsion parameter $\beta = 10$. Figure~\ref{fig:fig1} displays the error decay curves and energy and eigenvalue approximation behaviors for both methods. As shown in Figures~\ref{fig:fig1:1} and \ref{fig:fig1:2}, the $EQ_1^{\mathrm{rot}}$ element maintains the same convergence rates observed in the previous example, whereas the $Q_2$ element exhibits a higher convergence order, consistent with expectations. Crucially, Figure~\ref{fig:fig1:3} reveals that the $Q_2$ element approximates the exact energy from above, while the $EQ_1^{\mathrm{rot}}$ element provides an approximation from below. This numerical evidence strongly validates the conclusion of Theorem~\ref{thm:lowerEnergy}. In practical computations, combining the upper bound from conforming elements with the lower bound from nonconforming elements allows for the construction of a reliable confidence interval for the exact ground state energy, which a single conforming method cannot provide. Furthermore, Figure~\ref{fig:fig1:4} shows an interesting phenomenon: although our theoretical proof for the lower bound property is currently limited to the energy, the numerical results suggest that the discrete eigenvalue $\lambda_h$ also approaches the exact value from below. This implies that the lower bound property of nonconforming elements might be generalizable to the eigenvalues of the Gross--Pitaevskii equation under certain conditions.

        \begin{figure}[htb]
            \centering
            \begin{minipage}[c]{0.45\textwidth}
                \centering
                \includegraphics[width=1\textwidth]{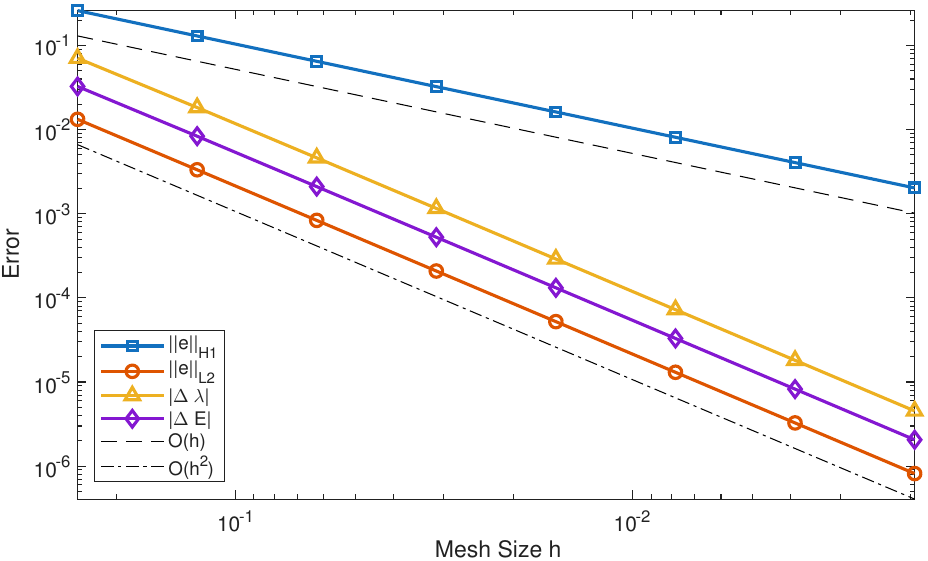}
                \subcaption{Convergence of $EQ_1^{\mathrm{rot}}$ element}
                \label{fig:fig1:1}
            \end{minipage}
            \begin{minipage}[c]{0.45\textwidth}
                \centering
                \includegraphics[width=1\textwidth]{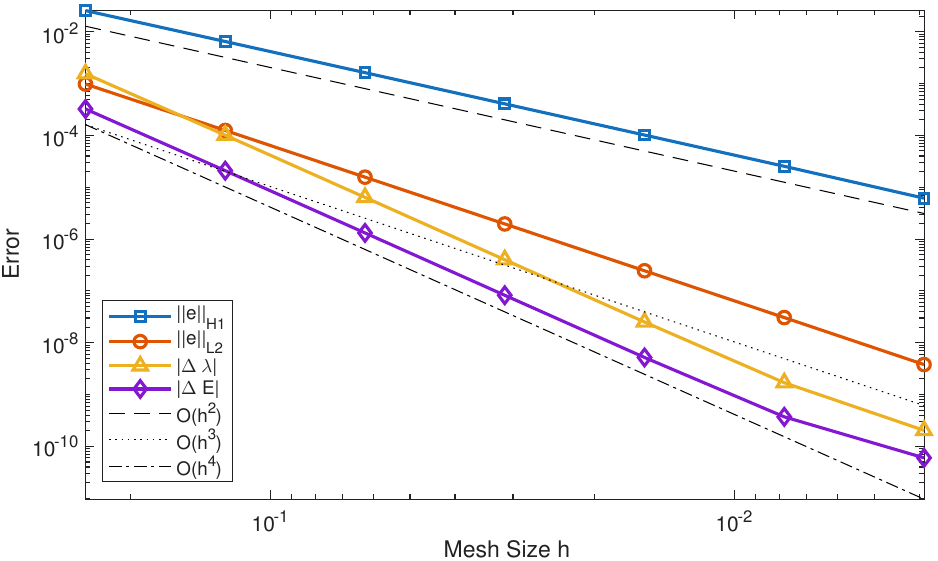}
                \subcaption{Convergence of $Q_2$ element}
                \label{fig:fig1:2}
            \end{minipage}
            \begin{minipage}[c]{0.45\textwidth}
                \centering
                \includegraphics[width=1\textwidth]{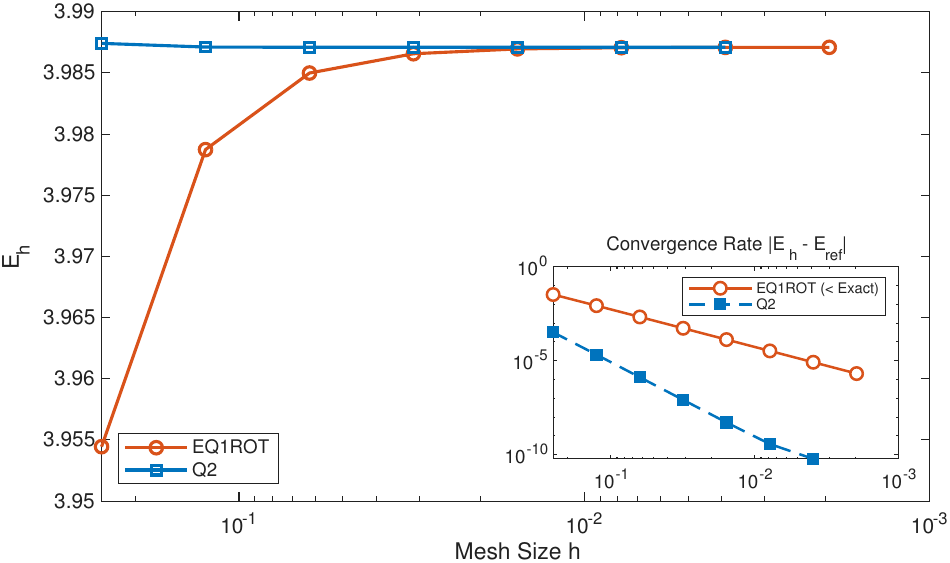}
                \subcaption{Evolution of energy functional $E$}
                \label{fig:fig1:3}
            \end{minipage}
            \begin{minipage}[c]{0.45\textwidth}
                \centering
                \includegraphics[width=1\textwidth]{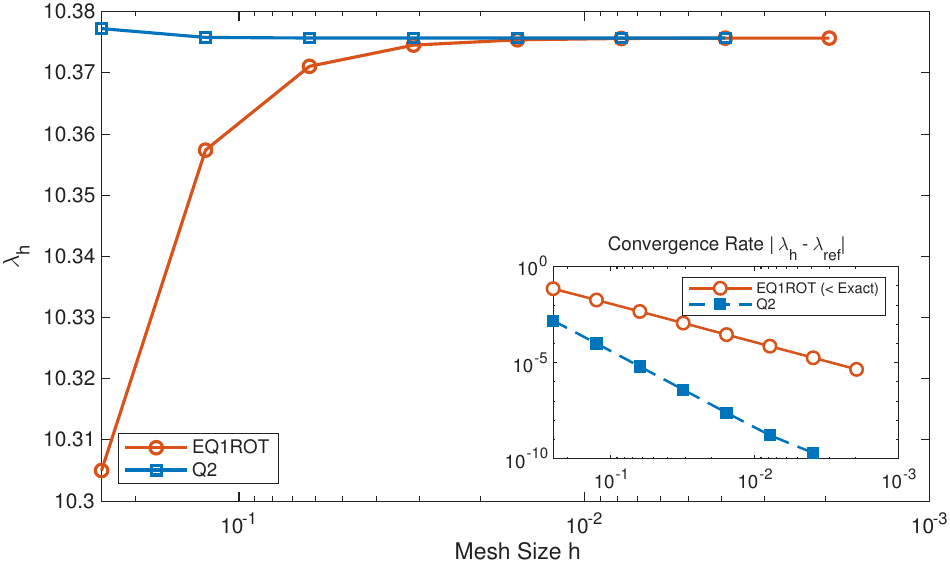}
                \subcaption{Evolution of eigenvalue $\lambda$}
                \label{fig:fig1:4}
            \end{minipage}
            \caption{Convergence results for Example~\ref{subsec:fig1}.}
            \label{fig:fig1}
        \end{figure}

        \subsection{Asymptotic convergence under complicated potentials}\label{subsec:fig2}

        Finally, we further evaluate the algorithm using Case II from \cite{bao2004computing}*{Example 3}. The domain is $\region = [-8, 8]^2$, with $\beta = 400$, and the trapping potential consists of an isotropic harmonic potential combined with a stirrer potential:
        \begin{equation*}
            V(x) = x_1^2 + x_2^2 + 8e^{-(x_1-1)^2-x_2^2}. 
        \end{equation*}
        The results are depicted in Figure~\ref{fig:fig2}. Notably, in contrast to Example~\ref{subsec:fig1}, the calculations on coarse meshes for both finite element spaces do not initially follow the theoretical convergence orders, nor does the $EQ_1^{\mathrm{rot}}$ element display the lower bound property for the discrete energy. The lower bound property only emerges once the mesh size $h$ is sufficiently small. This behavior aligns with the theoretical assumption in Theorem~\ref{thm:lowerEnergy} requiring $h$ to be "sufficiently small," indicating that the lower bound property holds asymptotically for certain problems. Additionally, in Figure~\ref{fig:fig2:1}, we observe that the eigenvalue $\lambda_h$ of the conforming $Q_2$ element does not strictly stay above the exact eigenvalue. This serves as a reminder that in nonlinear problems, while conforming elements guarantee an upper bound for the energy functional via the variational principle, their eigenvalues are influenced by both the nonlinear terms and mesh resolution, and thus are not as strictly guaranteed to approach from above as those of linear elliptic operators.

        \begin{figure}[htb]
            \centering
            \begin{minipage}[c]{0.45\textwidth}
                \centering
                \includegraphics[width=1\textwidth]{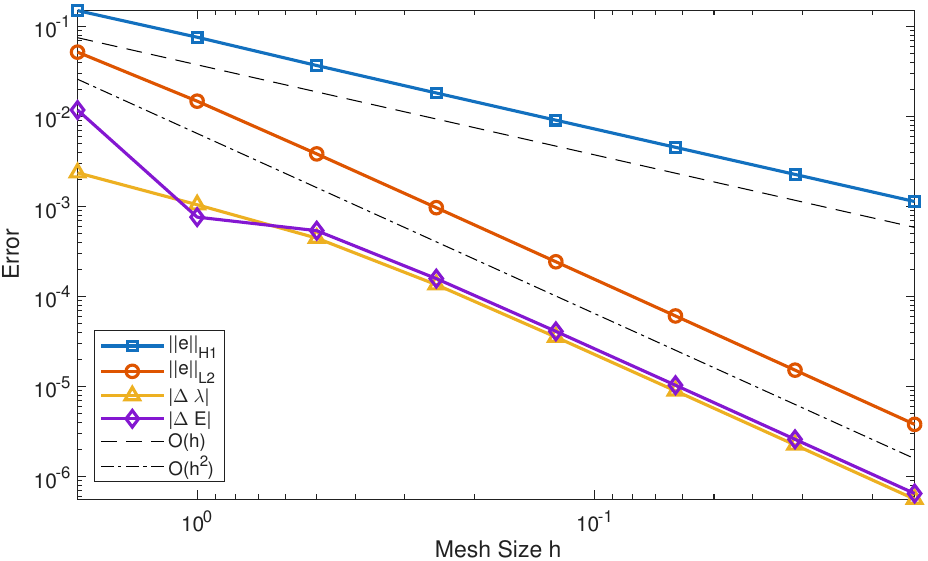}
                \subcaption{Convergence of $EQ_1^{\mathrm{rot}}$ element}
            \end{minipage}
            \begin{minipage}[c]{0.45\textwidth}
                \centering
                \includegraphics[width=1\textwidth]{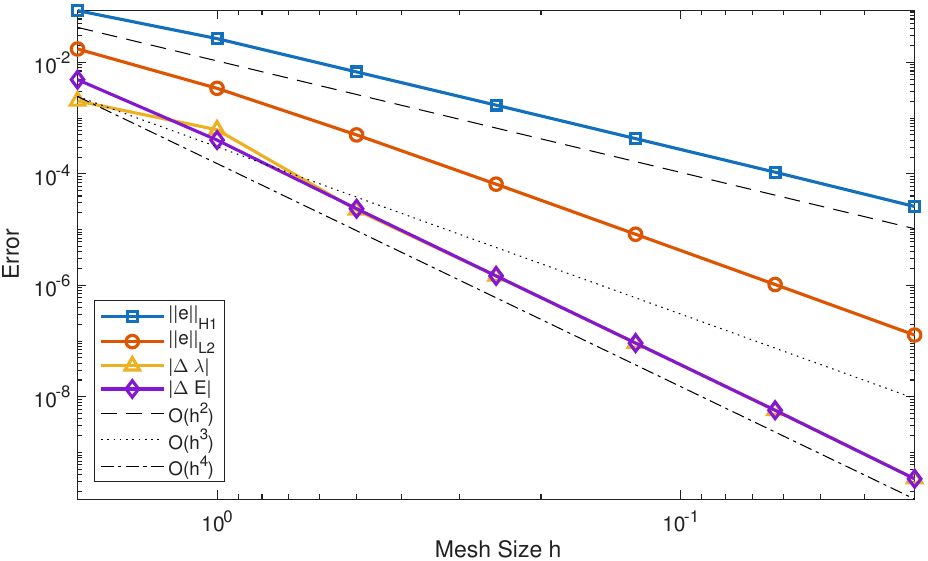}
                \subcaption{Convergence of $Q_2$ element}
            \end{minipage}
            \begin{minipage}[c]{0.45\textwidth}
                \centering
                \includegraphics[width=1\textwidth]{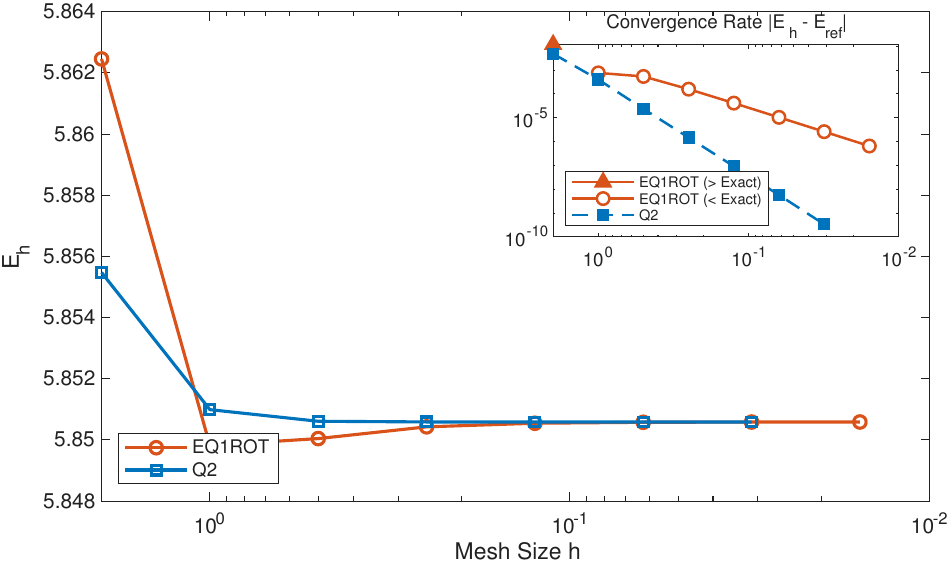}
                \subcaption{Evolution of energy functional $E$}
            \end{minipage}
            \begin{minipage}[c]{0.45\textwidth}
                \centering
                \includegraphics[width=1\textwidth]{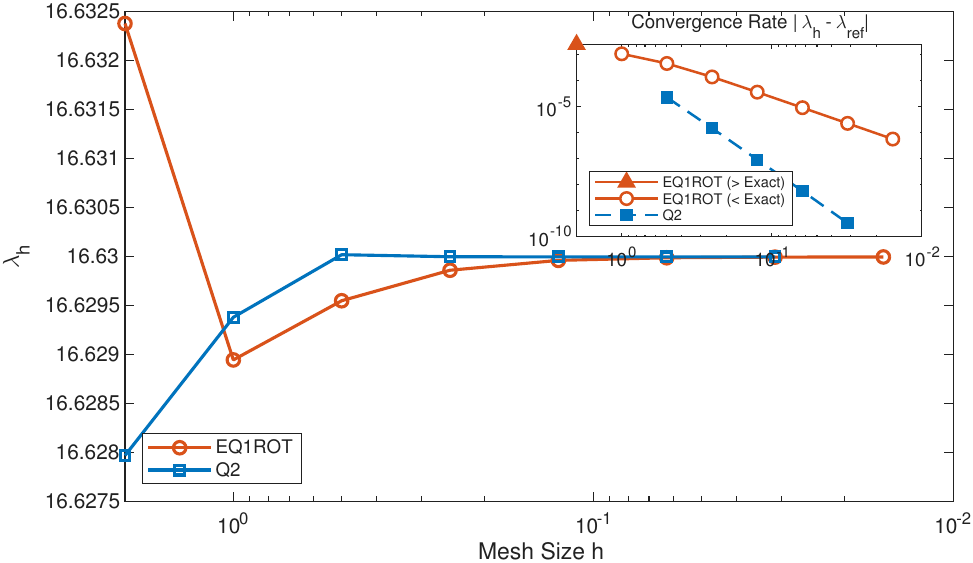}
                \subcaption{Evolution of eigenvalue $\lambda$}
                \label{fig:fig2:1}
            \end{minipage}
            \caption{Convergence results for Example~\ref{subsec:fig2}.}
            \label{fig:fig2}
        \end{figure}

    \section{Conclusions}

        In this work, we have studied the approximation properties of the constrained optimization problem~\eqref{eq:gp} for the Gross--Pitaevskii energy functional within the framework of nonconforming finite elements. We first established a general theoretical analysis framework, providing a priori error estimates for the discrete energy, the ground state eigenvalue, and the ground state solution in both $L^2$- and $H^1$-norms. Subsequently, this framework was applied to the $EQ_1^{\mathrm{rot}}$ element, where explicit convergence rates were derived. Moreover, we theoretically demonstrated that the $EQ_1^{\mathrm{rot}}$ element provides a lower bound approximation to the exact ground state energy when the mesh size is sufficiently small. This lower bound property serves as a complement to the upper bound estimates typically produced by conforming finite element methods. Finally, numerical results have validated our theoretical analysis, demonstrating the effectiveness of the $EQ_1^{\mathrm{rot}}$ element in solving this class of nonlinear eigenvalue problems.

    \subsection*{Acknowledgements}

        Chen is supported by the National Natural Science Foundation of China (NSFC 12471369 and NSFC 12241101). We are deeply grateful to the reviewers for their constructive comments and suggestions, which have led to improvements in this paper. In particular, their feedback inspired us to propose Lemma~\ref{lem:saturation} and Corollary~\ref{coro:monotonicity}.

    \bibliographystyle{amsxport}
    \bibliography{reference}

    \appendix

    \section{Properties of discrete ground states}\label{sec:prepare}

        In this section, we present several lemmas concerning the finite element space $V_h$ and its associated discrete ground states $u_h$. Lemma~\ref{lem:itp} provides an extension of the Sobolev embedding theorem to the space $V_h + H_0^1(\region)$.

        \begin{lemma}\label{lem:itp}
            For any $v_h \in V_h + H_0^1(\region)$ and $2 \le p \le 6$, it holds that
            \begin{equation*}
                \norm{v_h}_{L^p(\region)} \lesssim \norm{v_h}_{H^1_h(\region)}.
            \end{equation*}
        \end{lemma}

        \begin{proof}
            We present the proof for the three-dimensional case ($d=3$); the case for $d=2$ can be obtained analogously. Let $v_h = v_1 + v_2 \in V_h + H_0^1(\region)$, where $v_1 \in V_h$ and $v_2 \in H_0^1(\region)$. According to Theorem 2.1 and Corollary 2.2 in \cite{lasis2003poincare}, we have
            \begin{equation*}
                \begin{aligned}
                    \norm{v_h}_{L^6(\region)} &\lesssim \abs{v_h}_{H^1_h(\region)} + \left( \sum_{e\in \mF_h} \norm{[v_h]}_{L^4(e)}^2 \right)^{\frac{1}{2}} = \abs{v_h}_{H^1_h(\region)} + \left( \sum_{e\in \mF_h} \norm{[v_1]}_{L^4(e)}^2 \right)^{\frac{1}{2}} \\
                    &\lesssim \abs{v_h}_{H^1_h(\region)} + \left( \sum_{e\in \mF_h} h^{-1}\norm{[v_1]}_{L^2(e)}^2 \right)^{\frac{1}{2}} = \abs{v_h}_{H^1_h(\region)} + \left( \sum_{e\in \mF_h} h^{-1}\norm{[v_h]}_{L^2(e)}^2 \right)^{\frac{1}{2}} \\
                    &\lesssim \abs{v_h}_{H^1_h(\region)} + \norm{v_h}_{H^1_h(\region)} \le 2\norm{v_h}_{H^1_h(\region)}. 
                \end{aligned}
            \end{equation*}
            where the second to last inequality follows from \eqref{eq:assumpt:4} in Assumption~\ref{assumpt:consist}. By the embedding $L^6(\region) \hookrightarrow L^p(\region)$, we obtain
            \begin{equation*}
                \norm{v_h}_{L^p(\region)} \lesssim \norm{v_h}_{L^6(\region)} \lesssim \norm{v_h}_{H^1_h(\region)}. 
            \end{equation*}
            This completes the proof.
        \end{proof}

        Similar to \cite{pietro2012mathematical}*{Theorem 5.6}, we establish a generalized Rellich--Kondrachov compact embedding theorem as follows.

        \begin{lemma}\label{lem:embd}
            If $\{v_h\}$ is a sequence in $\{V_h + H_0^1(\region)\}$ that is uniformly bounded in the $H^1_h$-norm, then $\{v_h\}$ is precompact in $L^2(\region)$ and $L^4(\region)$.
        \end{lemma}

        \begin{proof}
            Consider the Bounded Variation (BV) norm~\cite{pietro2012mathematical}:
            \begin{equation*}
                \norm{v}_{BV}=\sup \left\{\int_{\region}v\left(\nabla\cdot\varphi\right) \diff r: \varphi\in C_0^{\infty}\left(\region; \R^d\right), \norm{\varphi}_{L^{\infty}(\region)}\le 1\right\}, 
            \end{equation*}
            We first demonstrate the existence of a mesh size $h_{\delta}>0$ such that the sequence $\{v_h\}_{h<h_{\delta}}$ is uniformly bounded in the BV-norm. For any $\varphi\in C_0^{\infty}(\region; \R^d)$ satisfying $\norm{\varphi}_{L^{\infty}(\region)}\le 1$, integration by parts yields
            \begin{equation*}
                \begin{aligned}
                    \int_{\region} v_h (\nabla\cdot \varphi) \diff r &= \sum_{K\in\mT_h} \int_{K} v_h (\nabla\cdot \varphi) \diff r = \sum_{K\in\mT_h} \left(-\int_K \varphi\cdot\nabla_h v_h \diff r + \int_{\partial K}v_h(\varphi\cdot n) \diff S\right) \\
                    &= -\int_{\region} \varphi\cdot\nabla_h v_h \diff r + \sum_{K\in\mT_h} \int_{\partial K}v_h(\varphi\cdot n) \diff S, 
                \end{aligned}
            \end{equation*}
            By the generalized patch test assumption~\eqref{eq:assumpt:3}, there exists $h_{\delta}>0$ such that for any $h<h_{\delta}$,
            \begin{equation*}
                \abs{\sum_{K\in\mT_h} \int_{\partial K}v_h(\varphi\cdot n) \diff S} \le 1, 
            \end{equation*}
            It then follows that
            \begin{equation*}
                \begin{aligned}
                    \norm{v_h}_{BV} &= \sup_{\norm{\varphi}_{L^{\infty}(\region)}\le 1} \int_{\region} v_h (\nabla\cdot \varphi) \diff r \\
                    &\le \sup_{\norm{\varphi}_{L^{\infty}(\region)}\le 1} \abs{\int_{\region} \varphi\cdot\nabla_h v_h \diff r} + \abs{\sum_{K\in\mT_h} \int_{\partial K}v_h(\varphi\cdot n) \diff S} \\
                    &\le \int_{\region} \abs{\nabla_h v_h} \diff r + 1 \le \sqrt{\abs{\region}}\norm{\nabla_h v_h}_{L^2(\region)} + 1 \\
                    &\le \sqrt{\abs{\region}}\norm{v_h}_{H^1_h(\region)} + 1 < \infty, 
                \end{aligned}
            \end{equation*}
            This indicates that $\{v_h\}_{h<h_{\delta}}$ is uniformly bounded in the BV-norm. According to \cite{ambrosio2000functions}*{Corollary 3.49}, the compact embedding $BV(\region)\hookrightarrow L^1(\region)$ implies that $\{v_h\}_{h<h_{\delta}}$ is precompact in $L^1(\region)$, which ensures the existence of a Cauchy subsequence $\{v_{h_j}\}_{j\in\N}$ in $L^1(\region)$. Furthermore, by Lemma~\ref{lem:itp}, the sequence $\{v_{h_j}\}_{j\in\N}$ is uniformly bounded in the $L^6$-norm. Utilizing the interpolation inequality, it can be shown that $\{v_{h_j}\}$ is a Cauchy sequence in $L^2(\region)$ and $L^4(\region)$, and consequently, it converges in these norms.

            In summary, $\{v_h\}$ is precompact in both $L^2(\region)$ and $L^4(\region)$.
        \end{proof}

        Finally, we address the uniform boundedness of $\{u_h\}$.

        \begin{lemma}\label{lem:proj}
            Let the operator $P_{H^1,h}: H_0^1(\region)\to V_h$ be defined such that for any $v\in H_0^1(\region)$ and $v_h\in V_h$,
            \begin{equation*}
                (P_{H^1,h}v, v_h)_{H^1_h(\region)} = (v, v_h)_{H^1_h(\region)}.
            \end{equation*}
            Then, the following convergence holds:
            \begin{equation*}
                \lim_{h\to 0}\abs{E(u) - E(P_{H^1,h}u)} = 0.
            \end{equation*}
        \end{lemma}

        \begin{proof}
            It is straightforward to verify that $P_{H^1,h}$ is a well-defined continuous linear operator satisfying
            \begin{equation*}
                \norm{P_{H^1,h}v}_{H^1_h(\region)}\le \norm{v}_{H^1(\region)},\qquad \norm{v - P_{H^1,h}v}_{H^1_h(\region)} = \inf_{v_h\in V_h} \norm{v - v_h}_{H^1_h(\region)}. 
            \end{equation*}
            Estimating each term separately, we have
            \begin{align*}
                &\abs{\int_{\region} \left( \abs{\nabla u}^2 - \abs{\nabla_h P_{H^1,h}u}^2 \right) \diff r} \lesssim \norm{u}_{H^1(\region)}\norm{u - P_{H^1,h}u}_{H^1_h(\region)},\\
                &\abs{\int_{\region} V\left(\abs{u}^2 - \abs{P_{H^1,h}u}^2\right) \diff r} \lesssim \norm{u}_{H^1(\region)}\norm{u - P_{H^1,h}u}_{H^1_h(\region)},\\
                &\abs{\int_{\region} \beta\left(\abs{u}^4 - \abs{P_{H^1,h}u}^4\right) \diff r} \lesssim \norm{u}_{H^1(\region)}^3 \norm{u - P_{H^1,h}u}_{H^1_h(\region)}. 
            \end{align*}
            Consequently,
            \begin{equation*}
                \abs{E(u) - E(P_{H^1,h}u)} \lesssim \left( \norm{u}_{H^1_h(\region)} + \norm{u}_{H^1_h(\region)}^3 \right)\norm{u - P_{H^1,h}u}_{H^1_h(\region)}. 
            \end{equation*}
            Applying the approximation property~\eqref{eq:assumpt:1} from Assumption~\ref{assumpt:consist} and taking the limit $h \to 0$, we obtain
            \begin{equation*}
                \lim_{h\to 0}\abs{E(u) - E(P_{H^1,h}u)} = 0, 
            \end{equation*}
            The proof is complete.
        \end{proof}

        \begin{cor}\label{coro:bound}
            For a sufficiently small mesh size $h$, the discrete ground states $\{u_h\}$ are uniformly bounded in the $H^1_h$-norm.
        \end{cor}

        \begin{proof}
            Let $c_h = \norm{P_{H^1,h} u}_{L^2(\region)}$. It follows that $\lim\limits_{h\to 0} c_h = \norm{u}_{L^2(\region)} = 1$. Thus, for sufficiently small $h$, we have
            \begin{equation*}
                \abs{u_h}_{H^1_h(\region)} \le E(u_h) \le E\left(c_h^{-1}P_{H^1,h}u\right) \le \left(1-\abs{1-c_h}\right)^{-4}\left(E(u) + \abs{E(u)-E(P_{H^1,h} u)}\right) \le 2E(u). 
            \end{equation*}
            Combined with $\norm{u_h}_{L^2(\region)} = 1$, we conclude that
            \begin{equation*}
                \norm{u_h}_{H^1_h(\region)} \le \norm{u_h}_{L^2(\region)} + \abs{u_h}_{H^1_h(\region)} \le 1 + 2E(u). 
            \end{equation*}
            The proof is complete.
        \end{proof}

    \section{Verification of assumptions for the \texorpdfstring{$EQ_1^{\mathrm{rot}}$}{EQ1rot} element}\label{sec:eq1rot}

        In this section, we verify that the $EQ_1^{\mathrm{rot}}$ element satisfies Assumption~\ref{assumpt:consist}. The approximation property of the $EQ_1^{\mathrm{rot}}$ element is a direct consequence of \cite{li2008lower}*{Lemma 1}. We begin by proving the final condition~\eqref{eq:assumpt:4} specified in Assumption~\ref{assumpt:consist}.

        \begin{lemma}\label{lem:jump}
            For any $v_h\in V_h + H_0^1(\region)$ and $e\in \mF_h$, the following estimate holds:
            \begin{equation*}
                \left( \sum_{e\in \mF_h} h^{-1}\norm{[v_h]}_{L^2(e)}^2 \right)^{\frac{1}{2}} \lesssim \norm{v_h}_{H^1_h(\region)}, 
            \end{equation*}
        \end{lemma}

        \begin{proof}
            Consider a reference element $\hat{K}$ and a face $\hat{e}\subseteq \partial \hat{K}$. For any $\hat{v}\in H^1(\hat{K})$, let $\overline{\hat{v}}$ denote the integral average of $\hat{v}$ over the face $\hat{e}$, defined as
            \begin{equation*}
                \overline{\hat{v}} = \frac{1}{\abs{\hat{e}}}\int_{\hat{e}} \hat{v} \diff S. 
            \end{equation*}
            By the Poincar\'e inequality for boundary traces with zero mean~\cite{repin2019poincare}, we have 
            \begin{equation*}
                \norm{\hat{v} - \overline{\hat{v}}}_{L^2(\hat{e})} \lesssim \abs{\hat{v}}_{H^1(\hat{K})}. 
            \end{equation*}
            Mapping the reference element back to the physical element, we consider two cases for any $e\in \mF_h$. If $e = K_1\cap K_2$ is an interior interface, let $K_e$ denote the union of the two elements $K_1$ and $K_2$. It follows that
            \begin{equation*}
                \norm{[v_h]}_{L^2(e)} \le \norm{v_h|_{K_1} - \overline{v_h}}_{L^2(e)} + \norm{v_h|_{K_2} - \overline{v_h}}_{L^2(e)} \lesssim h^{\frac{1}{2}}\left( \abs{v_h}_{H^1(K_1)} + \abs{v_h}_{H^1(K_2)} \right) = h^{\frac{1}{2}}\abs{v_h}_{H^1_h(K_e)}. 
            \end{equation*}
            If $e\subseteq \partial \region$ is a boundary face, let $K_e = K$. Then
            \begin{equation}\label{eq:consist:7}
                \norm{[v_h]}_{L^2(e)} = \norm{v_h}_{L^2(e)} \lesssim h^{\frac{1}{2}}\abs{v_h}_{H^1(K)} = h^{\frac{1}{2}}\abs{v_h}_{H^1_h(K_e)}. 
            \end{equation}
            Summing over all $e\in \mF_h$, we obtain
            \begin{equation*}
                \left( \sum_{e\in \mF_h} h^{-1}\norm{[v_h]}_{L^2(e)}^2 \right)^{\frac{1}{2}} \lesssim \left( \sum_{e\in \mF_h}\abs{v_h}_{H^1_h(K_e)}^2 \right)^{\frac{1}{2}} \lesssim \norm{v_h}_{H^1_h(\region)}. 
            \end{equation*}
            This completes the proof. 
        \end{proof}

        Next, we examine the generalized patch test for the $EQ_1^{\mathrm{rot}}$ element. Estimates for the consistency terms are mentioned in \cite{rannacher1992simple}. Here, we provide a proof for the sake of completeness. 

        \begin{lemma}
            The $EQ_1^{\mathrm{rot}}$ element passes the generalized patch test. In general, for any $v_h\in V_h + H_0^1(\region)$ and $v\in H^1(\region; \R^d)$, it holds that
            \begin{equation*}
                \abs{\mE_h(v_h, v)} \lesssim h\norm{v}_{H^1(\region)}\norm{v_h}_{H^1_h(\region)}. 
            \end{equation*}
        \end{lemma}

        \begin{proof}
            By the definition of the $EQ_1^{\mathrm{rot}}$ element, we have
            \begin{equation}\label{eq:consist:4}
                \begin{aligned}
                    \mE_h(v_h, v) &= \sum_{K\in \mT_h} \int_{\partial K} v_h (v\cdot n) \diff S = \sum_{e\in \mF_h} \int_e [v_h] (v\cdot n) \diff S = \sum_{e\in \mF_h} \int_e [v_h] (v\cdot n - \overline{v\cdot n}) \diff S \\
                    &\le \sum_{e\in \mF_h} \norm{[v_h]}_{L^2(e)}\norm{v\cdot n - \overline{v\cdot n}}_{L^2(e)}, 
                \end{aligned}
            \end{equation}
            where $\overline{v\cdot n}$ denotes the average of $v\cdot n$ over the face $e$. Similar to the proof of Lemma~\ref{lem:jump}, we have
            \begin{equation}\label{eq:consist:6}
                \norm{v\cdot n - \overline{v\cdot n}}_{L^2(e)} \lesssim h^{\frac{1}{2}} \abs{v\cdot n}_{H^1(K_e)} \lesssim h^{\frac{1}{2}} \abs{v}_{H^1(K_e)}. 
            \end{equation}
            Substituting \eqref{eq:consist:6} and the estimate \eqref{eq:consist:7} from Lemma~\ref{lem:jump} into \eqref{eq:consist:4} yields
            \begin{equation*}
                \begin{aligned}
                    \abs{\mE_h(v_h, v)} &\le \sum_{e\in \mF_h} \norm{[v_h]}_{L^2(e)}\norm{v\cdot n - \overline{v\cdot n}}_{L^2(e)} \lesssim h \sum_{e\in \mF_h} \abs{v_h}_{H^1_h(K_e)}\abs{v}_{H^1(K_e)} \\
                    &\le h \left( \sum_{e\in \mF_h} \abs{v_h}_{H^1_h(K_e)}^2 \right)^{\frac{1}{2}} \left( \sum_{e\in \mF_h} \abs{v}_{H^1(K_e)}^2 \right)^{\frac{1}{2}} \lesssim h \norm{v_h}_{H^1_h(\region)} \norm{v}_{H^1(\region)}. 
                \end{aligned}
            \end{equation*}
            The proof is thus complete. 
        \end{proof}

\end{document}